\documentclass[AMA,STIX1COL]{WileyNJD-v2}

\usepackage{algorithm}
\usepackage{algpseudocode}
\usepackage{caption}
\usepackage[font=normalsize]{subcaption}
\usepackage{graphicx}

\newcommand{\fracpartial}[2]{\frac{\partial #1}{\partial  #2}}
\newcolumntype{M}[1]{>{\centering\arraybackslash}m{#1}}
\newcolumntype{P}[1]{>{\centering\arraybackslash}p{#1}}

\newcolumntype{P}[1]{>{\centering\arraybackslash}p{#1}}

\newcommand{\beq}{\begin{equation}}
\newcommand{\eeq}{\end{equation}}
\newcommand{\beqr}{\begin{eqnarray}}
\newcommand{\eeqr}{\end{eqnarray}}

\newtheorem{Rem}{Remark}

\newcommand{\Htwo}{$\mathcal{H}_2$ }

\usepackage{wasysym}

\articletype{Research Article}%

\received{19 December 2022}
\revised{XXXX}
\accepted{YYYY}

\begin{document}

\title{Controlled Descent Training \thanks{This work has partially been supported by the project OCTON I-II at Chalmers University of Technology.}}

\author[1,2]{Viktor Andersson*}

\author[2]{Bal\'azs Varga}

\author[1,2]{Vincent Szolnoky}

\author[]{Andreas Syrén}

\author[2]{Rebecka Jörnsten} 

\author[2]{Bal\'azs Kulcs\'ar*}

\authormark{Viktor Andersson \textsc{et al}}

\address[1]{\orgdiv{Centiro AB}, \orgaddress{\state{Borås}, \country{Sweden}}}

\address[2]{\orgdiv{Chalmers University of Technology}, \orgaddress{\state{Gothenburg}, \country{Sweden}}}



\corres{*Viktor Andersson \email{vikta@chalmers.se}, *Bal\'azs Kulcsa\'ar  \email{kulcsar@chalmers.se}}


\abstract[Summary]{ In this work, a novel and model-based artificial neural network (ANN) training method is developed supported by optimal control theory. The method augments training labels in order to robustly guarantee training loss convergence and improve training convergence rate. Dynamic label augmentation is proposed within the framework of gradient descent training where the convergence of training loss is controlled. First, we capture the training behavior with the help of empirical Neural Tangent Kernels (NTK) and borrow tools from systems and control theory to analyze both the local and global training dynamics (e.g. stability, reachability). Second, we propose to dynamically alter the gradient descent training mechanism via fictitious labels as control inputs and an optimal state feedback policy. In this way, we enforce locally \Htwo optimal and convergent training behavior. The novel algorithm, \textit{Controlled Descent Training} (CDT), guarantees local convergence. 
CDT unleashes new potentials in the analysis, interpretation, and design of ANN architectures. The applicability of the method is demonstrated on standard regression and classification problems.} 



\maketitle


\section{Introduction}
Machine learning (ML) and Artificial Intelligence (AI) are able to model complex and highly non-linear input-output relationships. ML is very powerful but often lacks the guarantees and predictability required for system and control theory applications. 




Deep Artificial Neural Networks (ANNs) are particularly useful tools in machine learning that are commonly trained with gradient descent methods (GD). ANN architectures have made  strides in recent years in solving complicated tasks, like image recognition \cite{tripathi2021cnn}, natural language processing \cite{lauriola2022nlp}, artificial image generation \cite{aggarwal2021gans}, and other engineering tasks where classical statistical models might struggle. Problems in the domain of Control Theory have also had a surge in ANN and ML related works \cite{wan2020modelfree}\cite{maiworm2020guassian}\cite{marvi2020saferein}. The main criticisms of ANNs are their unpredictable training behavior, low output interpretability, and highly hyper-parameter dependent performance.

Making ANN learning behavior more predictable, with guarantees of convergence and reduced hyper-parameter search space would allow these powerful tools to be used in online systems and control applications. Tackling ANNs from a control systems perspective bridges a gap between the two disciplines and gives access to mathematically well grounded methods and tools such as stability and reachability analysis for ANNs.

Recent insights into the learning behavior of ANNs come from the study of ANNs with infinite width. \cite{jacot} introduces the Neural Tangent Kernel (NTK) for ANNs, describing the gradient descent training behavior as a linear ordinary difference equation (ODE). The authors use a first-order Taylor linearization of the NTK to derive a linear ODE description of the training for finite-width ANNs. Follow-up work by \cite{googlepaper} demonstrates the empirical region of validity of the Taylor linearization. The higher-order terms have been studied in \cite{huang2020hierarchy}. Tangential work studied the infinite NTK for many different architectures like CNNs \cite{googlepaper}, RNNs \cite{alemohammad2020recurrent} and transformers \cite{ntktransformers}. Moreover, \cite{yang2021tensor} introduces the notion of \emph{architectural universality} for the NTK and demonstrates its existence for any ANN architecture. The NTK has also been used for generative architectures. The internal stability of Generative Adversarial Networks (GANs) and other encoder-decoder ANN structures are revealed by the NTK framework \cite{franceschi2021gans}. Other works study the NTK for collegial ensemble methods \cite{Littwin2020ensambles}. The major criticism of the NTK is its computational complexity. \cite{zandieh2021scaling} introduce an NTK approximation method for improved computation speed of infinite NTKs. \cite{novak202fFast} significantly improves both computational times and reduce the memory required for finite NTK calculations by exploiting the Jacobian symmetric structure. This allows the NTK to be used for real-time training analysis.

The NTK has been used to explain aspects of ANN training dynamics already. For example, convergence properties of the NTK have been used for interpolation of phase transitions \cite{montanari2022interpolation}. \cite{nguyen2021bound} finds bounds on the smallest eigenvalue associated with memorization. \cite{jia2021test} suggest using the NTK to compare the informational discrepancy between test and training data sets. As demonstrated, the NTK is highly relevant to assess the training behavior of many ANNs architectures. The NTK description unlocks a more interpretable and dependable perspective on ANN training behavior. 

Optimal control has well-established analytic proofs of stability and optimality under conditions of reachability and open-loop stability\cite{kwakernaak1972linear}. Optimal control has made progress in recent years\cite{i402019Dolgui}, improving the reliability for high-dimensional and large scale systems\cite{largescaleOCHagebring}\cite{largescale2022Zhou}. Our main contribution is bringing these optimal control methods into the world of ML and ANNs.

In this paper, we define two categories of methods influencing the ANN training dynamic using the NTK. The first category is \emph{indirect} control through NTK reference and model matching. \cite{wanga2022PINNs} is an example of an implicit method, using the eigenvalues of the NTK to improve the convergence rate for physically informed neural networks (PINNs). The NTK has been used to complement existing policy gradient methods \cite{varga2021policy} and robust Q-learning \cite{varg2023deepQ}. The latter papers implement cascade optimal control methods to enhance the convergence and stability of reinforcement learning methods, implicitly.

In this paper, we follow a \emph{direct} approach to influence the ANN training dynamics using optimal control. Firstly, we introduce the concepts of stability and reachability for a given ANN, based purely on the NTK. This allows for directed and reduced hyper-parameter search spaces. Secondly, we develop a novel ANN training algorithm introduced as \emph{Controlled Descent Training} (CDT). CDT is a model-based,  \Htwo optimal state feedback label augmentation method built on the NTK that provides convergence guarantees (locally) and explicitly minimizes the cumulative training loss. This brings predictability of ANN learning, increased robustness to hyper-parameter choices, and guarantees otherwise missing for ANNs without compromising on performance.

We evaluate the performance of CDT compared to GD with different fully connected and convolutional ANN architectures. Both training algorithms are benchmarked on datasets selected from classic datasets (regression and classification). 
We demonstrate and report on the accuracy and numerical results of CDT algorithm with clear indications of its limitation.


The layout of the paper is as follows. In the Preliminaries (Section \ref{sec:prliminaries}), the NTK and the uncontrolled ANN training dynamics (borrowed from \cite{jacot}) are introduced and expanded on to describe the label augmented training dynamics.
Section 3 defines stability for the unaugmented and reachability for the augmented 
training dynamics.
In Section 4 the CDT algorithm is introduced.
Section 5 demonstrates the CDT algorithm performance compared with gradient descent for different architectures and ML problems.
We conclude the paper by highlighting future research directions. Additionally, the paper features multiple appendices elaborating on some of the assumptions and findings of the paper.

\section{Preliminaries}
\label{sec:prliminaries}
In this section, we briefly show how the Neural Tangent Kernel describes the training dynamics of an ANN.

We introduce the notation for an artificial neural network (ANN) as $F(\theta(k), x): \; \mathbb{R}^{r\times n_0} \rightarrow \mathbb{R}^{r\times n_L}$ where $\theta(k) \in \mathcal{C}$ denotes the parametrization comprising weights and biases at training step $k$. $\mathcal{C}$ here denotes at least once continuous differentiability. Let the output of the ANN for a fixed $n_0$-dimensional set of data $x \in \mathbb{R}^{r\times n_0}$ be $\hat y(\theta(k), x) = F(\theta(k), x) \in \mathbb{R}^{r\times n_L}$. Assume it is continuously differentiable with respect to $\theta(k)$: $\hat y(\theta(k), x)\in \mathcal{C}$.
$r$ denotes the number of data points.

\subsection{Neural Tangent Kernel}
The NTK is adopted to describe the ANN evolution in function space during gradient descent training\cite{jacot}. As the ANN width increases, the NTK evolution rate stagnates. For finite-width ANNs the NTK is time-varying resulting in a nonlinear Ordinary Difference Equation (ODE).

This latter NTK is referred to as the \emph{empirical tangent kernel}. The indirectly time-varying NTK ODE can be approximated as a time-invariant system by Taylor linearization around the initial parameters \cite{googlepaper}. Linearization of the finite width NTK ODE description paves the way for our main contribution, analysis, and explicit control of the ANN training dynamics. As such we use the empirical NTK and define it as follows. 
\begin{definition}
\textbf{Neural Tangent Kernel.} Given two data points $x_i, x_j \in \mathbb{R}^{ n_0} \subset x \in \mathcal{D}$, the NTK for an $r$ -batch size, $n_0$-input $n_L$-output ANN $F(\theta(k), x): \; \mathbb{R}^{r\times n_0} \rightarrow \mathbb{R}^{r\times n_L}$ at time instance $k \in \mathbb{Z}_+$, is 
\begin{equation}
    \Theta_{i,j}(k) =\left(\fracpartial{\hat{y}(\theta(k),x_i)}{\theta(k)}\fracpartial{\hat{y}(\theta(k),x_j)}{\theta(k)}^T\right) \in \mathbb{R}^{n_L\times n_L} \label{empiricalNTK}
\end{equation}
where $\hat y(\theta(k), x)$ is the output of the ANN. \\ 
We define the full NTK for a subset of data $x\in\mathcal{D}\subseteq \mathbb{R}^{n_0\times r}$ as
\begin{equation}
    \Theta(k)=\begin{bmatrix}
    \Theta_{0,0}(k) & \Theta_{1,0}(k) & \hdots & \Theta_{r,0}(k)\\
    \Theta_{0,1}(k) & \Theta_{1,1}(k) & & \vdots \\
    \vdots & & \ddots & \vdots
    \\
    \Theta_{0,r}(k) & \hdots & \hdots & \Theta_{r,r}(k)
    \end{bmatrix} \in \mathbb{R}^{rn_L\times rn_L}
    \label{eq:kernel}
\end{equation}
\end{definition}
The above mentioned empirical Neural Tangent Kernel $\Theta(k)$ in eq. \eqref{eq:kernel} is always symmetric and positive semidefinite. Positive-definiteness of the NTK ensures the convergence of the loss to a minimum for a class of loss functions (e.g., quadratic losses) \cite{jacot}. A weak assumption for positive-definiteness can be made if each pair of training inputs are not parallel and lie within a Euclidean unit ball \cite{chen2020generalized}. Some additional conditions guaranteeing its definiteness are given in\cite{yegeometry}.

\subsection{Local and global ANN training dynamics}
In this section, with the help of \cite{jacot,googlepaper} local and global finite-width ANN training behavior is introduced.
 
Assuming a constant target vector $y\in \mathcal{Y} \subseteq \mathbb{R}^{n_L\times r}$ (i.e., static labels in supervised learning), the output follows certain dynamics dictated by gradient descent. 
For the sake of brevity, we denote $\hat{y}(\theta(k),x)$ as $\hat{y}(k)$, bearing in mind that the estimated output $\hat{y}(k)$ still depends on the input data sequence. 
Furthermore, we assume the loss function is at least once continuously differentiable $\mathcal{L}(\hat{y}(k), y) \in \mathcal{C}$ with respect to $\theta(k)$ and $\hat y(k)$ at any time instance $k\in \mathbb{Z}$.
The evolution of the parameter vector $\theta(k)$ and thereof the network output $\hat{y}(k)$ under gradient descent with learning-rate $\alpha$ is given by
\begin{eqnarray}
\label{eq:theta_update}
\theta(k+1) &=& \theta(k)-\alpha\fracpartial{\mathcal{L}(\hat{y}(k), y)}{\theta(k)}
= \theta(k)-\alpha 
\fracpartial{\hat{y}(k)^T }{\theta(k)}  \fracpartial{\mathcal{L}(\hat{y}(k), y)}{\hat{y}(k)} \\
\hat{y}(k+1) &=& \hat{y}(k)-\alpha
\Theta(k) \fracpartial{\mathcal{L}(\hat{y}(k),y)}{\hat{y}(k)}=f_{\theta}(\hat y(k))
\label{eq:outputODE}.
\end{eqnarray}
Eq.~\eqref{eq:outputODE} captures the evolution of the  global training dynamics as a nonlinear time-discrete (ODE).

As can be seen in eq. \eqref{eq:outputODE} the symmetric empirical kernel $\Theta(k)$ has a central role in describing the training behaviour. 


Proposition 1 in Appendix \ref{appx:lipschitzness} (taken from \cite{khalil2002nonlinear}) demonstrates that the global training under gradient descent has a unique solution on a discrete time interval $k\in[k_i,k_f]$.
A local and linear (in $\hat y$) training dynamics can be obtained at any time instance $k$, by first order Taylor series approximation of eq.\eqref{eq:outputODE} (see Appendix \ref{appx:linearization} for full derivation). In this case,  $\hat y(k)$ is approximated at $\theta(k_0)$ when $\vartheta(k) \equiv \theta(k)-\theta(k_0)$, such that
\begin{eqnarray}
\hat y_{\vartheta}(k+1) = 
 \hat y_{\vartheta}(k)-
\alpha\Theta(k_0) \fracpartial{\mathcal{L}(\hat{y}_\vartheta(k),y)}{\hat{y}_\vartheta(k)}=f_{\vartheta}(\hat y_\vartheta(k)).
\label{eq:linear}
\end{eqnarray}

A bound on the error between the local and global dynamics can be found using the Lagrangian error bound (see Appendix \ref{appx:Lagrangian}). This bound allows us to quantify the error introduced via the first order approximation. Additional linearization may be required for certain loss functions to reach input affine form (see Appendix \ref{appx:losslin}). 



\subsection{Controlled ANN training dynamics with label augmentation} \label{sec:label augments}
As mentioned previously, one of the main contributions of the paper is to explicitly control the NTK training dynamics. As such, we introduce a dynamic label augmentation method, i.e. inject fictitious, time dependent labels by $y_u(k)$ as
\begin{equation}
    \bar{y}(k) = y + y_u(k)
\end{equation}
Unlike $y$, $\bar{y}(k)$ dynamically alters the targets to be estimated by the ANN. The label augmented dynamics (controlled global training dynamics) is then formulated by,
\begin{eqnarray}
\hat{y}(k+1) &=& \hat{y}(k)-\alpha
\Theta(\theta(k)) \fracpartial{\mathcal{L}(\hat{y}(k),\bar y(k))}{\hat{y}(k)}.
\label{eq:outputODEmse}
\end{eqnarray}

In eq. \eqref{eq:outputODEmse}, the injected labels purports to give a new degree of freedom to influence the local and global training behavior. In Section \ref{sec:CDT} an optimal way to select the fictitious labels $y_u(k)$ is demonstrated. 


\section{Analytic properties of discrete-time training dynamics}
Before the CDT algorithm 
is introduced two conditions  are established under which CDT guarantees convergence of the local training dynamics; stability and reachability.

Firstly, global and local stability concepts of ANN training dynamics are defined around specific equilibrium values. Stability guarantees boundedness of the unaugmented training dynamics. 
Secondly, we analyze the local controlled training dynamics from a reachability perspective. If reachability conditions are met, this ensures that the label augments $y_u(k)$  can help us to reach any points in $\mathcal{Y}$ within a finite number of steps. Both properties can be verified before training using the initial kernel $\Theta_0$.

\subsection{Boundedness of the training dynamics}
We relate boundedness of the network output via eq. ~\eqref{eq:outputODE} to the context of internal stability. Stability refers to the existence of a finite bound between the ANN output $\hat{y}$ and some equilibrium output $\hat{y}_e$. Here, a training equilibrium point $\hat y_e$ is defined as follows,
\begin{eqnarray}
\hat y(k + 1) = \hat y(k) = \hat y_e.
\label{eq:equilibrium}
\end{eqnarray}

Furthermore, it follows from the dynamics in eq. ~\eqref{eq:outputODE} that for most conventional losses equilibrium points may exist at $\hat{y}_e=y$. We discuss the conditions under which an equilibrium point may exist in Appendix \ref{appx:equilibriums}. 
The formal stability definition of eq. ~\eqref{eq:outputODE} can be captured by the following definition.
\begin{definition}
\textbf{Uniform internal stability \cite{rugh}}
The discrete-time ANN training dynamics with network output $\hat y(k)$, initial network output $\hat y(k_0)=y_0$ and equilibrium point $\hat y_e$ is called uniformly bounded if there exists a finite positive constant $\gamma$ such that for any $k_0$ and $y_0$ the corresponding solution satisfies 
\begin{equation}
||\hat y(k)-\hat y_e||_2\leq \gamma ||\hat y(k_0)-\hat y_e||_2, \qquad k\geq k_0
\end{equation}
\end{definition}
In essence, the uniform stability guarantees the ANN output during training $\hat{y}$ does not diverge from the equilibrium point $\hat{y}_e$ to infinity in finite time. The stronger stability condition of exponential stability is defined as,
\begin{definition}
\textbf{Uniform exponential internal stability \cite{rugh}}
The discrete-time ANN training dynamics in eq. ~\eqref{eq:outputODE} with network output $\hat y(k)$, initial prediction $\hat y(k_0)=y_0$, and equilibrium point $\hat y_e$ is called uniformly exponentially stable if there exists a finite positive constant $\gamma$ and a constant $0< \kappa \leq 1$ such that for all $k_0$ and $y_0$

\begin{equation}
\label{eq:LinearizedExpBoundedness}
 ||\hat y(k)-\hat y_e||_2\leq \gamma \kappa^{k-k_0} ||\hat y(k_0) - \hat y_e||_2, \quad k\geq k_0.
\end{equation}
\end{definition}


To verify the above mentioned conditions for generic loss functions by using the global training dynamics is an uneasy task. However, the internal stability conditions of the local training dynamics described in eq. \eqref{eq:linear} may result in simplified conditions. As an example, the stability conditions of local training dynamics with quadratic loss reduce to an eigenvalue condition. As such, internal stability reads as 

\begin{equation}
    \label{eq:stabilityOfMSE}
    |\lambda| \leq 1, \quad det\left(\lambda I-\left(I-\alpha \Theta(k_0) \right)\right)=0 \quad 
    \forall \lambda 
\end{equation}
where $\alpha$ uniformly scales the eigenvalues of the local-empirical NTK, $\Theta(k_0)$. If $\Theta(k_0)$ is positive semi-definite and none of the scaled eigenvalues of $\alpha \Theta(k_0)$ is larger than 1, the inequality in eq.~\eqref{eq:stabilityOfMSE} is strict. Hence, $\alpha$ can be choosen such that the local training dynamics is guarantee to be stable stable. Furthermore, this guarantees in the local sense that the equilibrium output is asymptotically reached. 
In Appendix \ref{appx:boundednessLosses} a concise and loss function dependent derivation of stability analysis is provided for certain common loss functions. 
\begin{Rem} \textbf{Learning rate adaptation}.
Finally, some ANN training algorithms\cite{adaptiveLRTong} suggest altering the learning rate $\alpha$. Intuitively, the learning rate scales the eigenvalues of the local-empirical NTK and as such impacts stability. Modifying the scalar parameter $\alpha$ may help the convergence of the training dynamics.
\end{Rem}   

\begin{Rem} \textbf{Robust stability.}
    The linearization error discussed in Appendix \ref{appx:Lagrangian} can be propagated through eq.\eqref{eq:LinearizedExpBoundedness} to deduce global stability from the local dynamics. This is further discussed in Appendix \ref{appx:boundedness}.
\end{Rem}

\subsection{Reachability}
Reachability is a property of the label augmented training dynamics given in eq.\eqref{eq:outputODEmse}. It verifies the existence of a bounded sequence of the label augments $y_u(k)$ such that any targeted ANN output $\hat{y}(k_f)$ can be reached from $\hat y_0$ within $f$ finite steps. 
\begin{definition} \textbf{Reachability.}
The label augmented training dynamics eq.~\eqref{eq:outputODEmse} is called reachable on $[k_0, k_f]$ if from a given initial state $\hat y(k_0)$ there exists at a sequence of $y_u(k)$ such that any $\hat y (k_f)$ can be reached $k_0<k_f<\infty$. \label{def:reachability}
\end{definition}

From Definition \eqref{def:reachability} the label reachability (with $y(k_f)=y$) can be derived as a specific case of reachability. The reachability condition for the global training dynamics (with input affine label augments) in eq. \eqref{eq:outputODEmse} can be verified using difference-geometric \cite{Isidori1995NL} or set theoretic algorithms \cite{ReachabilityNL}. This also indicates that addressing the reachability question for generic and complex loss functions is hard.

In some specific cases of the loss function (e.g. if the controlled training dynamics is local and the label augments are injected in an input affine way), the reachability analysis is straightforward to perform. Especially, the reachability analysis of quadratic losses and local training dynamics can be concluded by using linear systems and control theory \cite{rugh}. In such cases, we borrow the Popov-Belman-Hautus (PBH) test given by, 
\begin{eqnarray}
rank\begin{bmatrix}
zI-(I-\alpha \Theta(k_0)) & \alpha \Theta(k_0)\end{bmatrix}=rn_L, \quad \forall z\in\mathbb{C}.
\label{eq:reachability}
\end{eqnarray}
Numerically, the PBH condition consists of the finitely many rank tests at the eigenvalues of $(I-\alpha \Theta(k_0))$. 
\begin{Rem} \textbf{Unreachable local training dynamics}.
The importance of reachability in ANN training can easily be captured when it is not full-filled. 
A specific example is if there exist two identical data points. In such a case, the local empirical NTK has two similar rows or columns (see eq. \ref{empiricalNTK}) causing rank deficiency in $\Theta(k_0)$. The intuitive explanation is that two similar or identical data points will yield the same ANN output. Hence these two points are inseparable in ANN output space. In practice, if these data points have the same label the training will be stabilizable (see remark \ref{rem:stabilizability}).   
\end{Rem}

\begin{Rem}\textbf{Overfitting.}
 Intuitively, if reachability is fulfilled, the augmented training can perfectly fit the training data, causing overfitting. Consequently, if reachability is not fulfilled the ANN cannot perfectly fit the data. Hence reachability can be a good measure of whether or not a network is complex enough to fit the data or if data has conflicting data points. Moreover, the overfitting caused by the augmented learning can be remedied with various regularization techniques (e.g. \cite{Vincent2022}).
 \end{Rem}

\begin{Rem} \textbf{Stabilizability.}
\label{rem:stabilizability}
If the local dynamics is not full state reachable but the non-reachable states partition is locally asymptotically stable, we call the training dynamics locally stabilizable.
\end{Rem}

Finally, the above mentioned analytic conditions (stability, reachability, stabilizability) support the deployment of model based and optimal  label augmentation solutions.

\section{Controlled Descent Training - Locally Optimal Control of ANN training dynamics} \label{sec:CDT}

In Section \ref{sec:label augments}, label augments, as new fictitious inputs, have been injected into the training dynamics. In the following section, it is demonstrated how to calculate the label augments $y_u(k)$ such that stability and some optimality criteria are (at least locally) satisfied. The main idea is to use $\hat y_{\vartheta}(k)$ and transform it to $y_u(k)$ with a static gain. In Figure \ref{fig:blockdiagram}, the schematics of the closed-loop and controlled label augmentation for an network trained with MSE is depicted.

\begin{figure}[!htp]
\centering
\includegraphics[trim={0 5.5cm 3cm 0},clip, width=0.8\textwidth]{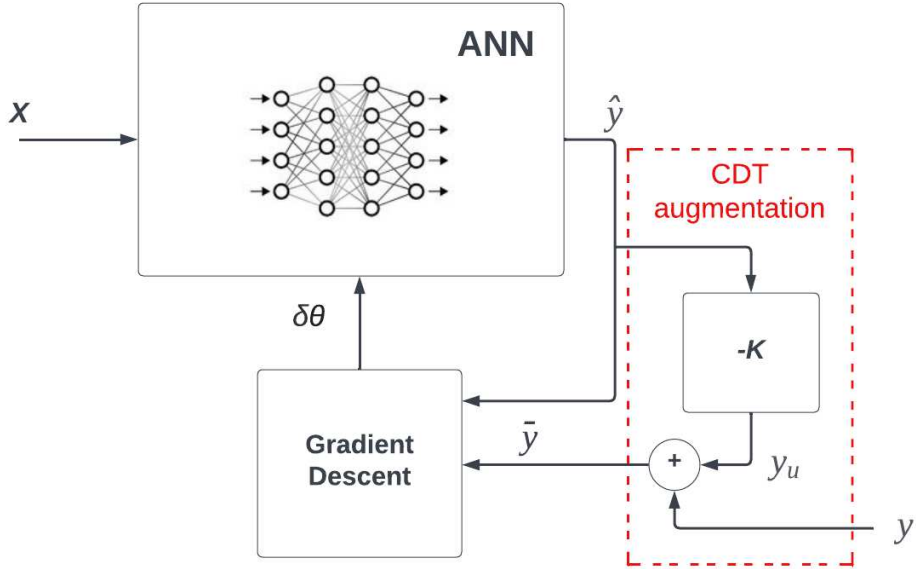}
         \caption{Schematic of Controlled Decent Training. The CDT augmentation block calculates the label augment $y_u$ dynamically from the ANN output $\hat y$ using the controller feedback matrix $K$. The new dynamically augmented label $\bar y=y+y_u$ is fed to GD rather than the static label $y$.}
         \label{fig:blockdiag}
\label{fig:blockdiagram}
\end{figure}
In order to find $K$ (in Figure \ref{fig:blockdiagram}), we propose to use an optimal state feedback label augmentation method.  More precisely, $y_u(k)$ label injection is aimed at \Htwo optimal closed loop training dynamics (CDT). In the following section, we restrict ourselves to quadratic loss functions and assume the augmented training dynamics is stabilizable (or reachable).

The following notation is introduced,
\begin{eqnarray}
 & \tilde{y}(k)  = \begin{bmatrix}
    \hat y_{\vartheta}(k) \\ 1
\end{bmatrix},\\
& 
\end{eqnarray}

This allows the standard infinite horizon cost to account for the offset introduced by the labels $y$. More precisely, the following infinite horizon cost is minimized according to
\begin{eqnarray}
& \min_{y_u} \frac 12 \sum _{i=k_0}^{\infty}\tilde y(i)^T \tilde Q\tilde y(i)+ y_u(i)^TRy_u(i) \label{eq:LQC1} \\
& s.t. \ \tilde{y}(k+1) = \begin{bmatrix}     I-\alpha\Theta(k_0) & \alpha\Theta(k_0) y \\
     \mathbf{0} & 1
\end{bmatrix}\tilde{y}(k) + \begin{bmatrix}
     \alpha\Theta(k_0)\\ \textbf{0}
 \end{bmatrix}y_u(k) \label{eq:LQC2} \\
\label{eq:LQC}
\end{eqnarray}

where $\tilde Q\in\mathbb{R}^{r(n_L+1)\times r(n_L+1)}$ and $R\in\mathbb{R}_+^{rn_L\times rn_L}$ are real valued positive semi-definite and positive definite weighting matrices, respectively. More precisely, 
\begin{eqnarray}
\tilde Q=\begin{bmatrix} Q & -Qy\\ -y^TQ & y^TQy \end{bmatrix}.
\end{eqnarray}
The cost includes the weighted squared error between the ANN predictions according to the linear dynamics and the targets, as well as the weighted square sum of the label augment $y_u$. The weighting matrix $Q\in\mathbb{R}^{rn_L\times rn_L}$ can be chosen such that certain data points or ANN outputs are more important than others.  Moreover, the local training dynamics in eq.~\eqref{eq:LQC2} captures learning interactions between data points in $x$ which in turn influence the optimal solution. The optimization problem, if solved, delivers an \Htwo optimal label augmentation solution. The cost function in eq.~\eqref{eq:LQC1} describes a generic energy approach to label augment selection where the weighting matrices $Q,R$ shape their relative importance. Finally, the first term in eq.~\eqref{eq:LQC1} penalizes the deviation from the static targets.

The locally stabilizing and optimal solution to eq.~\eqref{eq:LQC1}-\eqref{eq:LQC} can be found by using the Discrete-time Algebraic Riccati equations (DARE)  \footnote{Stabilizability and detectability conditions must hold \cite{kwakernaak1972linear}}(see Appendix \ref{appx:DARE} for solution derivation). If the stationary and extremal solution to DARE is $P$ then the optimal feedback gain can be written as 
\begin{eqnarray}
 K&=&
\left (R+\begin{bmatrix}
     \alpha\Theta(k_0)\\ \textbf{0}
 \end{bmatrix}^TP\begin{bmatrix}
     \alpha\Theta(k_0)\\ \textbf{0}
 \end{bmatrix} \right)^{-1}\begin{bmatrix}
     \alpha\Theta(k_0)\\ \textbf{0}
 \end{bmatrix}^TP\begin{bmatrix}     I-\alpha\Theta(k_0) & \alpha\Theta(k_0) y \\
     \mathbf{0} & 1
\end{bmatrix} \label{eq:control-feedback}\\
y_u(k)&=&-K\tilde{y}(k) \\ 
\bar y(k)&=&y_u(k)+y=-K\tilde y(k)+y\\
\end{eqnarray}



The closed and CDT controlled loop becomes  
\begin{eqnarray}
\begin{bmatrix}
    \hat y_{\vartheta}(k+1)\\1
\end{bmatrix}=\left(\begin{bmatrix}     I-\alpha\Theta(k_0) & \alpha\Theta(k_0) y \\
     \mathbf{0} & 1
\end{bmatrix} - \begin{bmatrix}
     \alpha\Theta(k_0)\\ \textbf{0}
 \end{bmatrix}K \right)\begin{bmatrix}
    \hat y_{\vartheta}(k)\\1
\end{bmatrix}.
\label{eq:cloop}
\end{eqnarray}

\noindent The feedback gain matrix $K$ maps the ANN output $\tilde y$ to \textit{target augments} $y_u(k)$ such that it minimizes the cost in eq.~\eqref{eq:LQC} on an infinite horizon. Note, the optimal cost value with the state feedback policy is $\frac 12 \tilde y(k_0)^TP\tilde y(k_0)$.
Finally, the controller gain $K$ can be calculated before training and remains constant during training.
In the local dynamical sense, the linear difference equation in eq. \eqref{eq:cloop} guarantees asymptotic stability, and therefore convergence. 
\begin{Rem} \textbf{Batch}.
CDT suggests using the local empirical NTK for the whole training dataset (megabatch). In practice, it may be more attractive with a traditional mini-batch approach, recalculating the NTK and feedback controller for each batch. This would call for receding horizon optimal control. 


\end{Rem}
\begin{Rem} \textbf{Robustness}.
The proposed \Htwo state feedback control policy is robust with a guaranteed magnitude \cite{kwakernaak1972linear}. This makes CDT applicable on the global training dynamics in practice. However, for proper handling of the modeling error between the global and the local training dynamics, robust control methods are proposed.
\end{Rem}

\subsection{The CDT algorithm}
In previous sections, the concepts of reachability and stability were introduced for ANNs and their implications on hyper-parameter selection examined. The optimal target augment sequence $\bar{y}(k)$ was calculated using LQR such that stability is guaranteed and convergence rate improved.
The full CD training algorithm for MSE is summarized in Algorithm \ref{alg:cdt}.

\begin{algorithm}
\caption{CDT summary}\label{alg:cdt}
\begin{algorithmic}
\State 1: \quad Calculate $\Theta_0 \gets$ eq.\eqref{eq:kernel} \Comment{Calculate kernel at ANN initialization}
\State 2: \quad Check stability with eq.\eqref{eq:stabilityOfMSE}
\Comment{(Stability)}
\State 3: \quad Check reachability with eq.\eqref{eq:reachability}
\Comment{(Reachability)}

\State 4: \quad Calculate $K \gets$ eq.\eqref{eq:control-feedback} \Comment{Calculate control feedback using DARE} \\
\textbf{ANN Training:}
\For{$k\in1,2,...,$ \textit{training steps}}
    \State 1: \quad Calculate $\bar y(k) \gets y-K\tilde y(k)$ \Comment{Calculate label augments based on feedback error}
    \State 2: \quad Update parameters $\theta(k+1) \gets \theta(k) - \delta\mathcal{L}(\hat y(k), \bar y(k))/ \delta \theta(k)$ \Comment{Update parameters according to GD}
\EndFor
\end{algorithmic}
\end{algorithm}

\section{Experiments}

In this section, traditional gradient descent (GD) and CDT are compared numerically using two standard benchmarking datasets. The first example is a regression problem using the Ames Housing dataset \cite{regdataset} with a single-target fully connected ANN and MSE loss. 
The second example is a binary image classification problem using ALEXNet\cite{alexnet} for the purpose of demonstrating the applicability of CDT on Convolutional Neural Networks (CNN). 
Both experiments run on a megabatch setup, meaning the data is shuffled and split into validation and train datasets with all train data in a single batch. The loss is averaged over the batch. Each model is trained 10 times with reshuffled data and a new random initialization. For each dataset, the model is trained using both optimization methods for a number of learning rates all with learning rate decay according to 
\begin{eqnarray}
    \alpha_k = \frac{1}{1+0.01k}\alpha
\end{eqnarray}
where $k$ is the training step and $\alpha$ is the initial learning rate. The learning rate decay is not modeled in the training dynamics to ensure the controller does not compensate for the decay by scaling the system. The controller design cost matrices $Q$ and $R$ in eq.\eqref{eq:LQC} are chosen as scaled diagonal matrices,
\begin{eqnarray}
    Q = I_1 \\
    R = pI_2
\end{eqnarray}
with identity matrices $I_1 \in \mathbb{R}^{rn_L \times rn_L}$ and $I_2 \in \mathbb{R}^{r \times r}$ and a pre-selected control input cost $p$ \footnote{These penalties weights are tuning parameters}. Smaller values of be $p$ give larger label augment values.
 In the following experiments, $p$ is a constant ($p=0.1$) in order to demonstrate that this design parameter is significantly less sensitive than learning-rate $\alpha$.  Note however that choosing $p$ will influence performance. There are multiple heuristics involving the choice of $p$, any of which can be used to yield even lower validation loss. However, this paper focuses on demonstrating the applicability of the method and theory, rather than performance improvement. 
The experiments are written in Python (3.7) using the latest version (1.9.1) of PyTorch released by Facebooks AI Research Lab 2016.

\subsection{Regression}

The Ames Housing Price dataset contains 79 explanatory variables describing houses in Iowa along with their final sale price. The regression target for this dataset is the sale price. A description for each variable can be seen in \cite{regdataset}. The full dataset has 2919 entries. For the experiments, 512 data points are sampled without replacement, normalized around 0 and split into 70\% training and 30\% validation data. The experiments are run on a mega batch setup meaning all training data is run concurrently in a single batch. \footnote{For traditional mini-batch gradient descent the control scheme would be recalculated for each batch, analogous to a receding horizon controller or MPC. For the purposes of this paper, the mega-batch setup better demonstrates the theory presented.}

\subsubsection{Architecture}
For the regression experiment 3 ANN architectures similar to the model description used in \cite{googlepaper} is used with initializations given in Appendix \ref{appx:init}. I.e. a fully connected feed-forward neural network setup is used according to,
\begin{eqnarray}
    \begin{cases}
        h^{l+1} &=z^l W^{l+1} + b^{l+1} \\
        z^{l+1} &= \phi^l(h^{l+1})
    \end{cases}
    \qquad
    \begin{cases}
        W^l_{i,j} &= \frac{\sigma_w}{\sqrt{n_{l}}} w^l_{i,j} \\
        b^l_j &= \sigma_b \beta^l_j
    \end{cases}
\end{eqnarray}
where $l<L\in \mathbb{N}$ is the layer where $L$ is the final layer, $n_l$ is number of input features to the layer $l$, $z^0$ is the input data $x\in\mathcal{D} \subseteq \mathbb{R}^{n_0\times r}$ and $r$ is the batch size. $w_{i,j}$ and $\beta_j$ is the weight and bias where $i=1,\hdots,n_l$ and $j=1,\hdots,n_{l+1}$. $W^l$ and $b^l$ are the matrix and vector describing the weights and bias of a layer respectively. $\phi_l$ is the output activation function for layer $l$, hence $h$ is the output from layer $l$ and $z$ is the input to the next layer $l+1$. For the regression set up no final activation $\phi^L$  function is used hence $z^L = h^L$. Finally $\hat{y}=z^L$ is the ANN output. The weights and biases are initialized with a normal distribution respectively $\mathcal{N}(0, \frac{\sigma_w^2}{n_l})$ and $\mathcal{N}(0, \sigma_b^2)$. We define a parameter vector as
\begin{eqnarray}
    \theta = \begin{bmatrix} 
    vec(W^L) & b^L & \hdots & vec(W^1) & b^1
    \end{bmatrix}^T
\end{eqnarray}
where $vec$ means concatenated vector form of the matrix. $\hat{y}(\theta,x) : \mathbb{R}^{n_0\times r} \xrightarrow{} \mathbb{R}^r$ is the output of the ANN using input data $x$ and parameters $\theta$.

The fully connected ANN architectures (of varying widths and depths) used for the Housing price dataset can be seen in Table \ref{tab:st-arch}. All architectures use ReLU as inter layer activation function with no final activation. This ensures the results are not architecture dependent and demonstrate how CDT is influenced by varying widths and depths.

\subsubsection{Regression experiment}

The experiment is run 10 times with different initializations, reshuffled training data, and validation indices for each architecture. Tables comparing the analytical and observed properties of the two training algorithms can be seen in Appendix K \footnote{In accordance to the Journal guidelines provided.} Tables \ref{table:a1} to \ref{table:a3}. 
In the aforementioned tables, $\alpha$ is the learning rate. Final validation loss is the average model performance over all initializations on the validation data at the final training iteration. For the purpose of demonstration, if some but not all ANN initializations resulted in divergent training the average loss over all non divergent initializations is indicated. The convergence column describes how many initializations resulted in non divergent training ($\hat y$ does not tend toward infinity). $|eig(\Theta(k_0))|<1$ describes the open-loop local stability of training for all initializations. 
Reachability describes the reachability of training for all initializations.  
Figures \ref{fig:a1_relative} to \ref{fig:a3_relative} show the average difference  ($\mathcal{L}_{GD}-\mathcal{L}_{CDT}$) of MSE validation loss between GD and CDT during training over all initializations for each architecture.  The relative difference between the two validation losses is always in favor of CDT (the metric is never $<0$) hence Figures \ref{fig:a1_relative} to \ref{fig:a3_relative} are shown in $log_{10}$ scale. Only learning rates where both CDT and GD converge for all initializations are shown in the figures. The absolute MSE losses for each architecture and learning rate can be seen in Figures \ref{fig:a0_1} to \ref{fig:a2_0001} in Appendix \ref{appx:figs}.

\begin{figure}[h]
    \hfill
     \begin{subfigure}{0.33\textwidth}
         \centering
         \includegraphics[width=\textwidth]{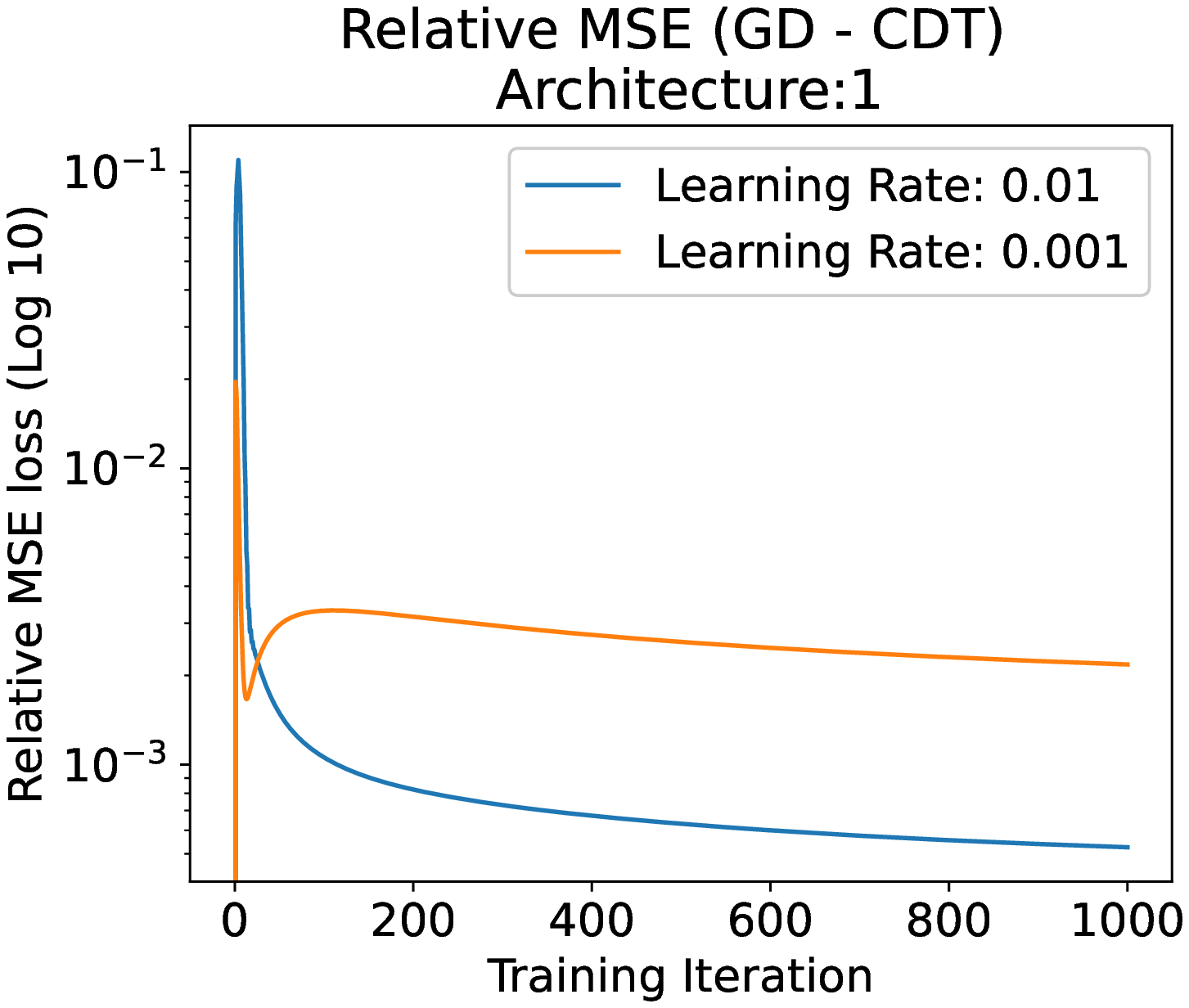}
         \caption{Architecture 1}
         \label{fig:a1_relative}
     \end{subfigure}
     \hfill
     \begin{subfigure}{0.33\textwidth}
         \centering
         \includegraphics[width=\textwidth]{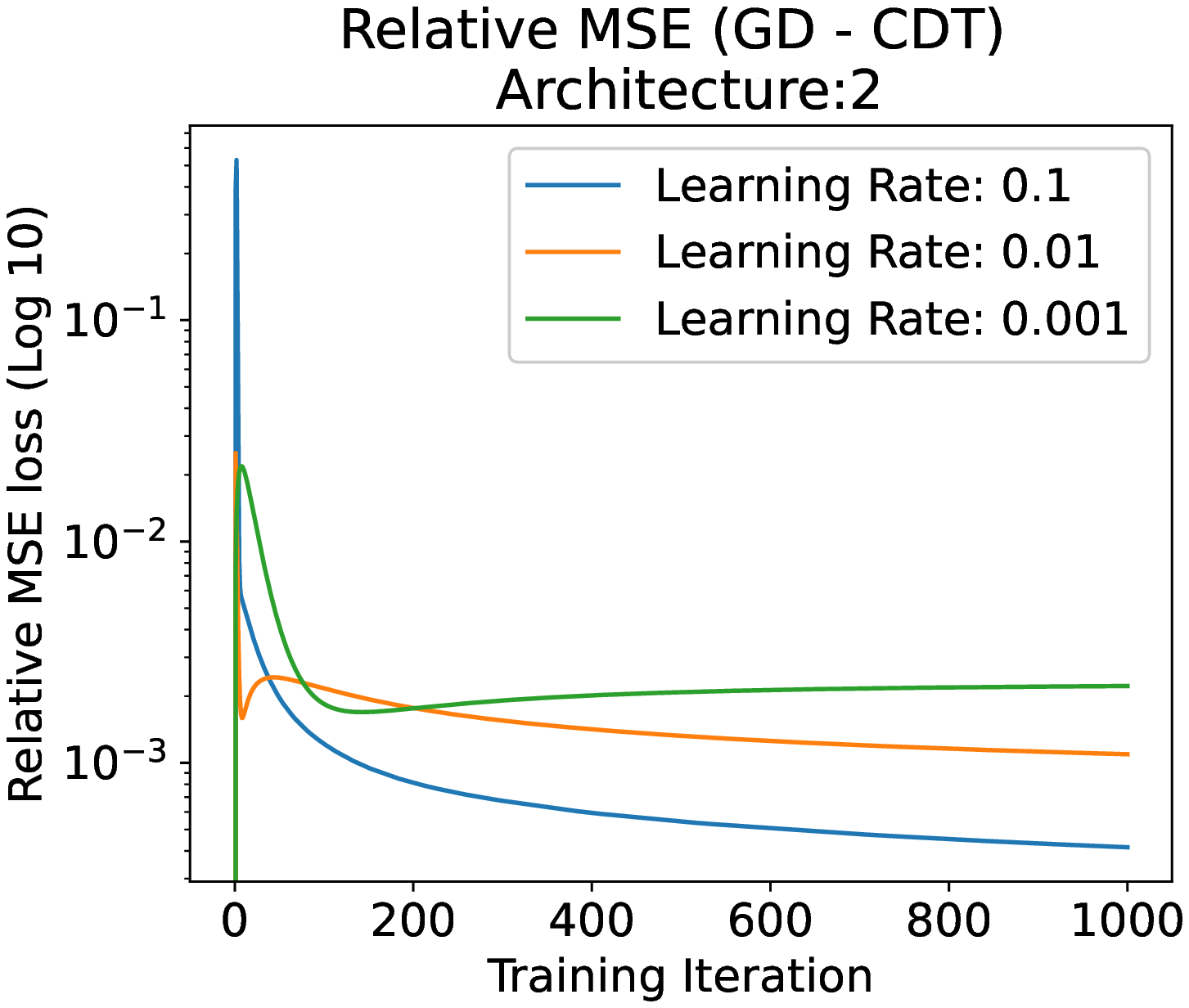}
         \caption{Architecture 2}
         \label{fig:a2_relative}
     \end{subfigure}
    \centering
      \begin{subfigure}{0.33\textwidth}
         \centering
         \includegraphics[width=\textwidth]{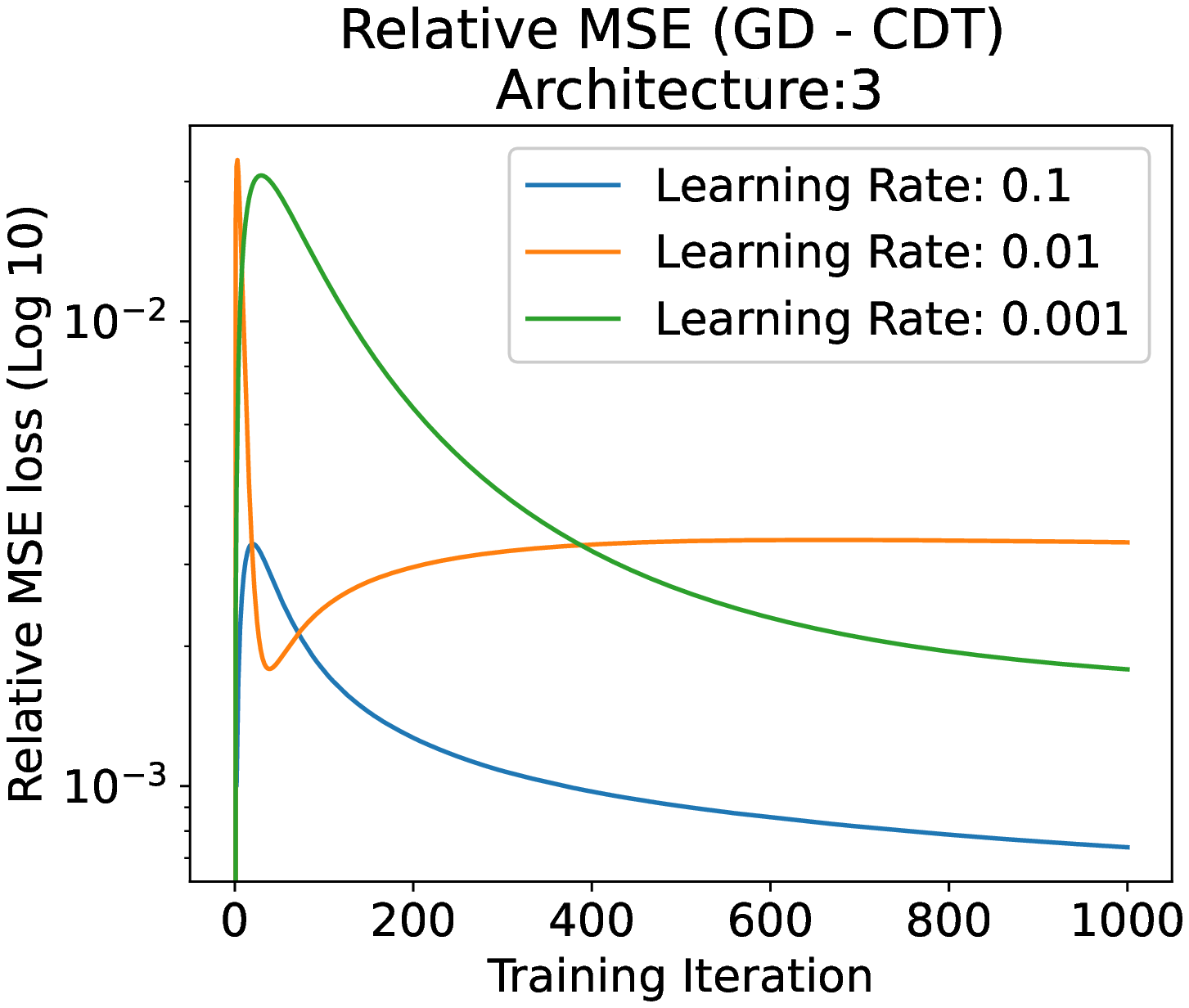}
         \caption{Architecture 3}
         \label{fig:a3_relative}
     \end{subfigure}
     \caption{Average relative difference between GD and CDTs ($\mathcal{L}_{GD}-\mathcal{L}_{CDT}$) MSE validation loss for all learning-rates where both training methods converge for each of the three regression architectures. CDT has a lower validation loss during the entirety of training for all learning rates and architectures.}
     \label{fig:relative}
\end{figure}

\begin{figure}
\centering
         \begin{subfigure}{0.32\textwidth}
         \centering
         \includegraphics[width=\textwidth]{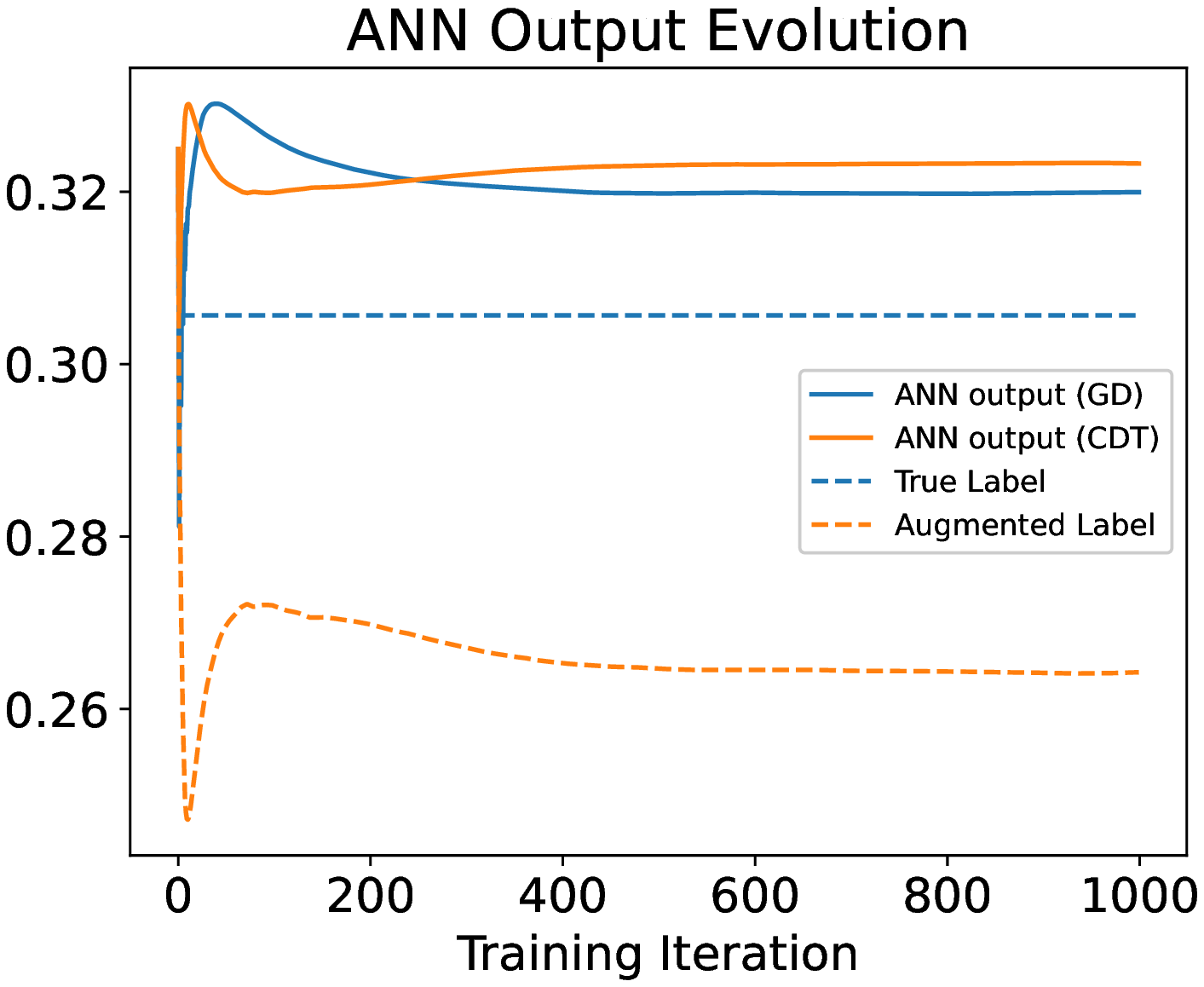}
         \caption{Sample 1}
         \label{fig:s1}
     \end{subfigure}
     \begin{subfigure}{0.32\textwidth}
         \centering
         \includegraphics[width=\textwidth]{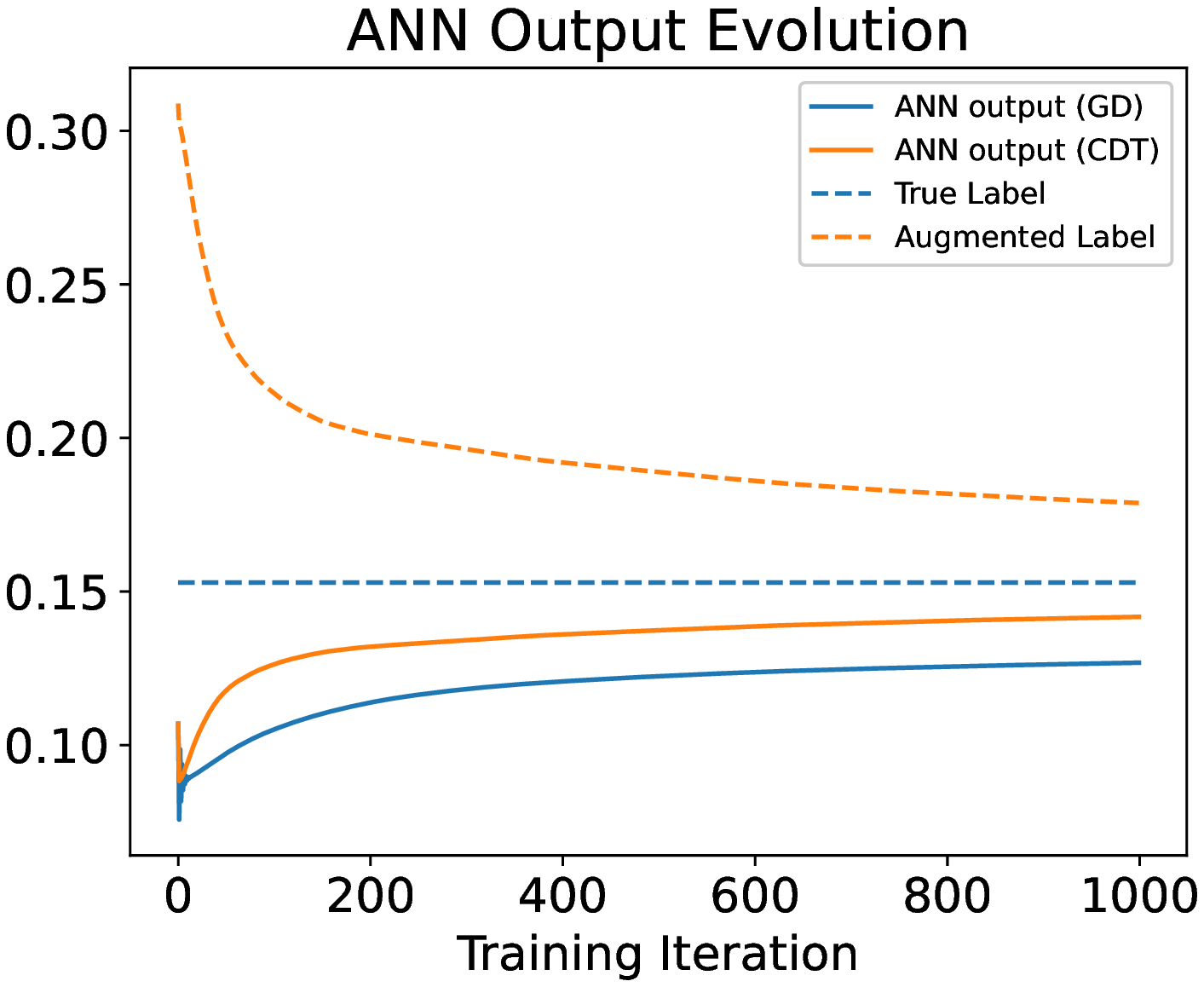}
        \caption{Sample 2}
         \label{fig:s2}
     \end{subfigure}
     \begin{subfigure}{0.32\textwidth}
         \centering
         \includegraphics[width=\textwidth]{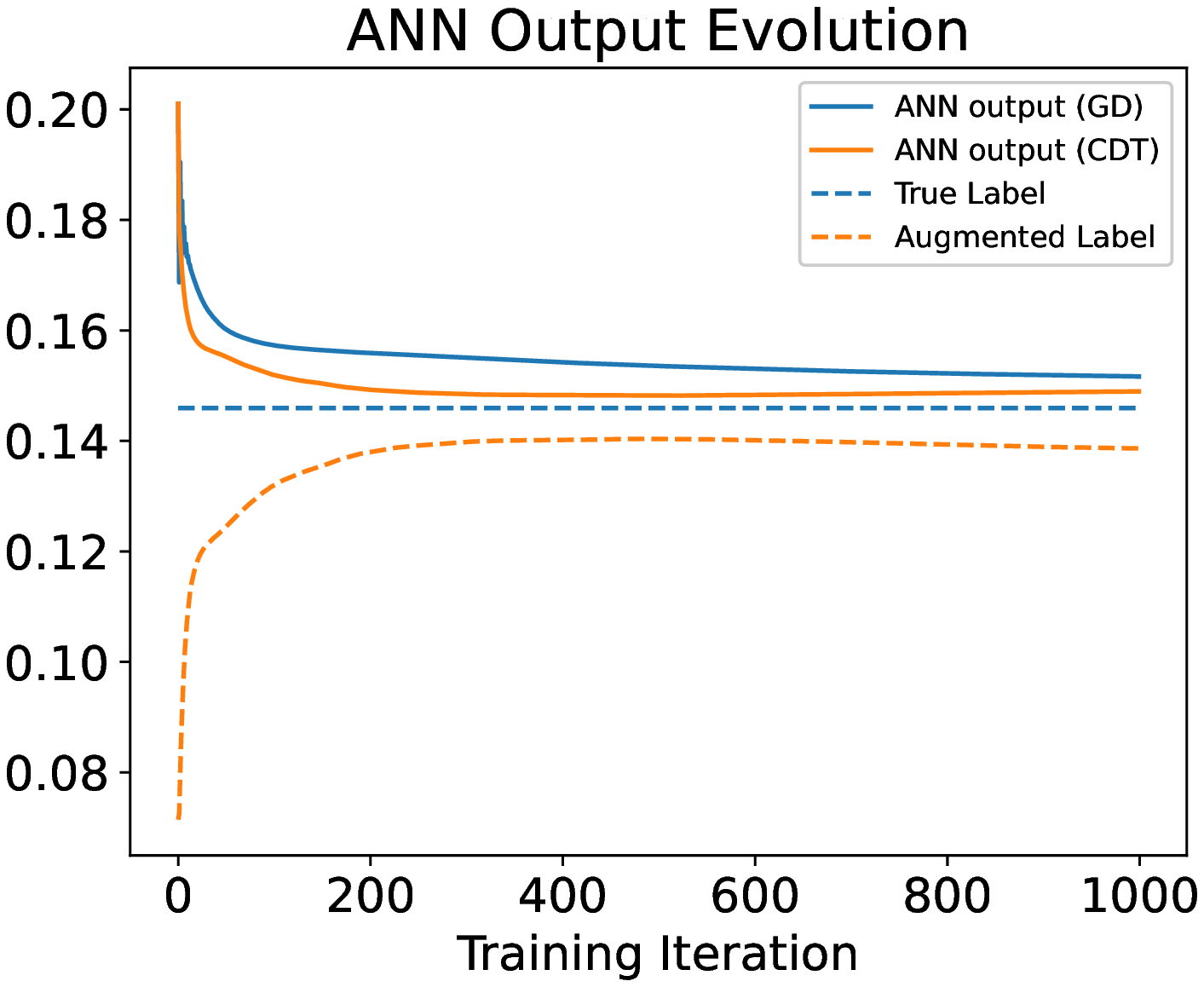}
         \caption{Sample 3}
        \label{fig:s3}
     \end{subfigure}
     \caption{ANN outputs $\hat y$, true label $y$ and augmented label $\bar{y}$ evolution for random samples when training under CDT and GD. First initialization of architecture 1 with learning rate 0.01. CDT appear to converge closer to true label $y$ for most, but not all data samples in the training batch.}
    \label{fig:samples}
\end{figure}

As can be seen in Table \ref{table:a1} to \ref{table:a3}, CDT is more robust to higher learning-rates with competitive final MSE validation loss while SGD diverges to infinity. CDT consistently converges to a lower loss for all architectures and learning rates. Moreover, the CDT standard deviation is smaller than for GD hence the augmented training is more consistent between initializations and data shuffles. The difference between the highest and lowest final MSE loss is consistently smaller for CDT and the performance only changes significantly for very low learning rates. Hence, CDT is seemingly less affected by choice of learning rate $\alpha$ than traditional GD. This behavior is expected as the controller may scale the system as required. It can be seen in Table \ref{table:a1} that architecture 1 trained with CDT does not diverge for any initialization at the highest learning rate $\alpha=1.000$ but converges further away from the true labels than at initialization. Due to the high learning rate and relatively few parameters in the single hidden layer architecture, the true kernel $\Theta(k)$ changes rapidly making the global dynamics drift from the local approximation $\Theta(k_0)$. Note however that despite this CDT does not diverge to infinity.

Regarding observed global convergence it can be seen in Tables \ref{table:a1} to \ref{table:a3} that for some learning rates, the local stability condition is not fulfilled. Despite this, both GD and CDT are observed to converge to the true labels. This hints at the higher order interactions not modeled by the first order Taylor approximation improves robustness and does not cause divergent training. This requires more analysis to confirm and is left for future work. Furthermore, the reachability condition is always fulfilled for all architectures. The dataset provided is clean and thoroughly examined for duplicates and other issues hence loss of reachability resulting from duplicated data points is not an issue.


Figure \ref{fig:a1_relative} to \ref{fig:a3_relative} demonstrate that CDT converges in fewer iterations compared to GD. The difference in performance is largest at the start of training, meaning CDT has already converged to a lower loss than GD during early iterations. Figures \ref{fig:a0_01} to \ref{fig:a2_0001} in Appendix \ref{appx:figs} highlight this further. 

Figure \ref{fig:samples} show the ANN output evolution for 3 randomly selected samples and a single initialization under both training methods along with the augmented $\bar y(k)$ and static $y$ labels. Not that in Figure\ref{fig:s1} CDT converges further away from the true label than GD. For this initialization, the ANN trained with CDT is closer to the true labels for 253 out of 357 samples in the training batch at the final iteration. Since CDT gives a lower average loss and outputs closer to the true targets on most samples but not all, it can be concluded that some data samples are prioritized by the CDT method while others are not. As stated previously the empirical kernel $\Theta(k_0)$ is a matrix describing the effect of each sample on all other samples during training\cite{jacot}. Hence the CDT algorithm will prioritize samples with a large effect on others such that minimal loss is achieved. Since $K$ is a static linear transform of the ANN output as $\hat y$ approaches the true static labels $y$ the augmented label $\bar y$ converges to the true static label $y$.

\subsection{Classification experiments}
The Microsoft research Cats vs. Dogs dataset \cite{binarydataset} contains 25k images depicting cats and dogs equally distributed. However, for the demonstrative purposes of this work, a subset of 256 images are sampled without replacement. In order to verify the generalization properties of CDT 70\% of the data is placed in the validation set. The random sampling makes no distinction between the classes, therefore the sampled dataset is not balanced between the classes. Each image is resized to $96x96$ pixels with all color channels retained. The ALEXNet \cite{alexnet} CNN architecture is used for this dataset. This architecture is very complex compared to the previous regression example hence overfitting is expected. For the purposes of demonstration, the CNN is trained with multi-target MSE rather than the standard cross-entropy loss.\footnote{We do this for two reasons; (1) it is more closely connected to the theory presented which is the main focus of this paper, and (2) it is easier to verify the theory applicability on CNNs with multiple outputs without additional linearizations.} Figures \ref{fig:alexnet_1} to \ref{fig:alexnet_0001} show the MSE validation loss evolution of both CDT and GD for the different learning rates.

\begin{figure}[!htp]
    \centering
         \begin{subfigure}{0.34\textwidth}
         \caption{}
         \centering
         \includegraphics[width=\textwidth]{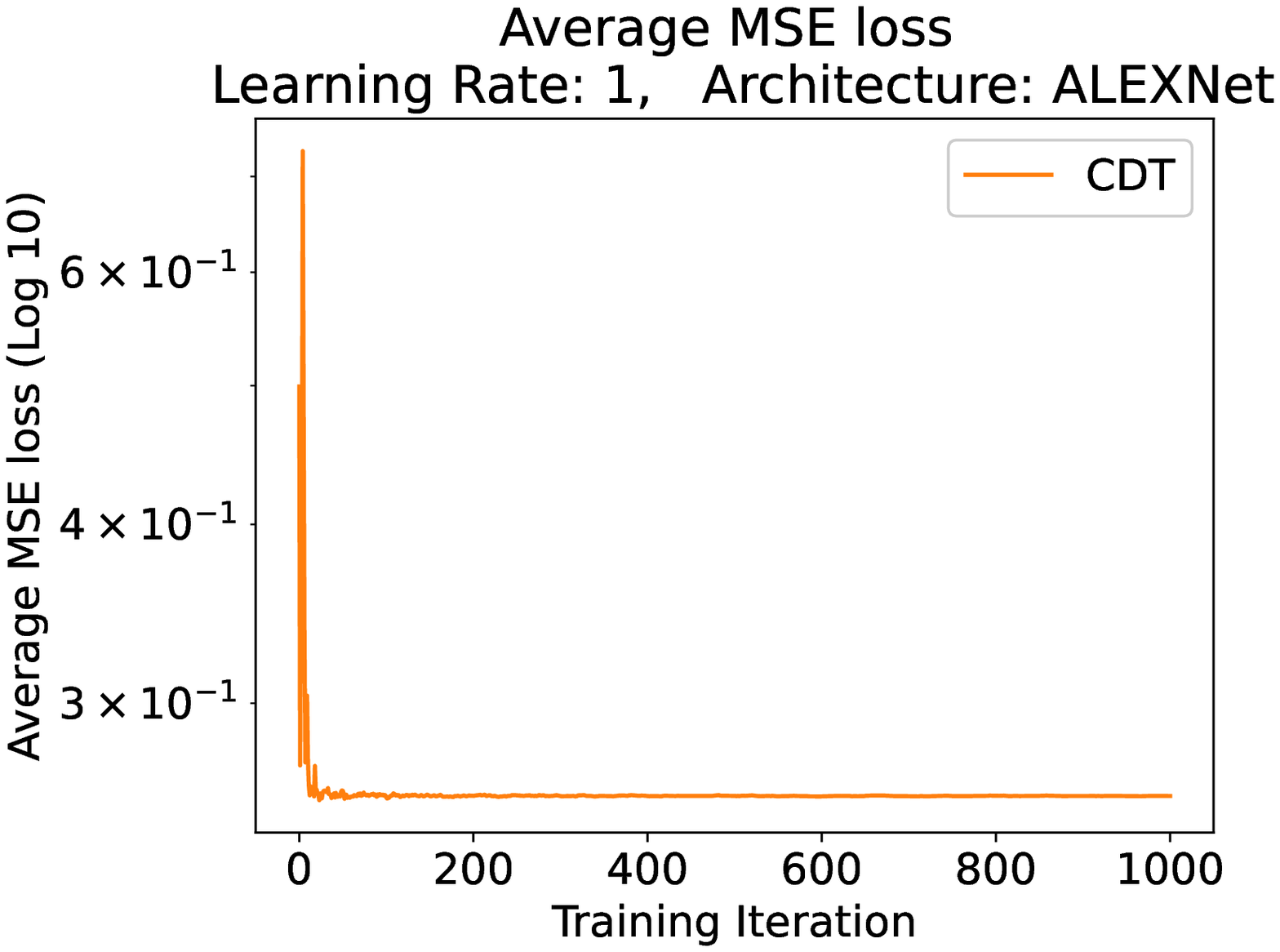}
         \label{fig:alexnet_1}
     \end{subfigure}
     \begin{subfigure}{0.34\textwidth}
         \centering
         \caption{}
         \includegraphics[width=\textwidth]{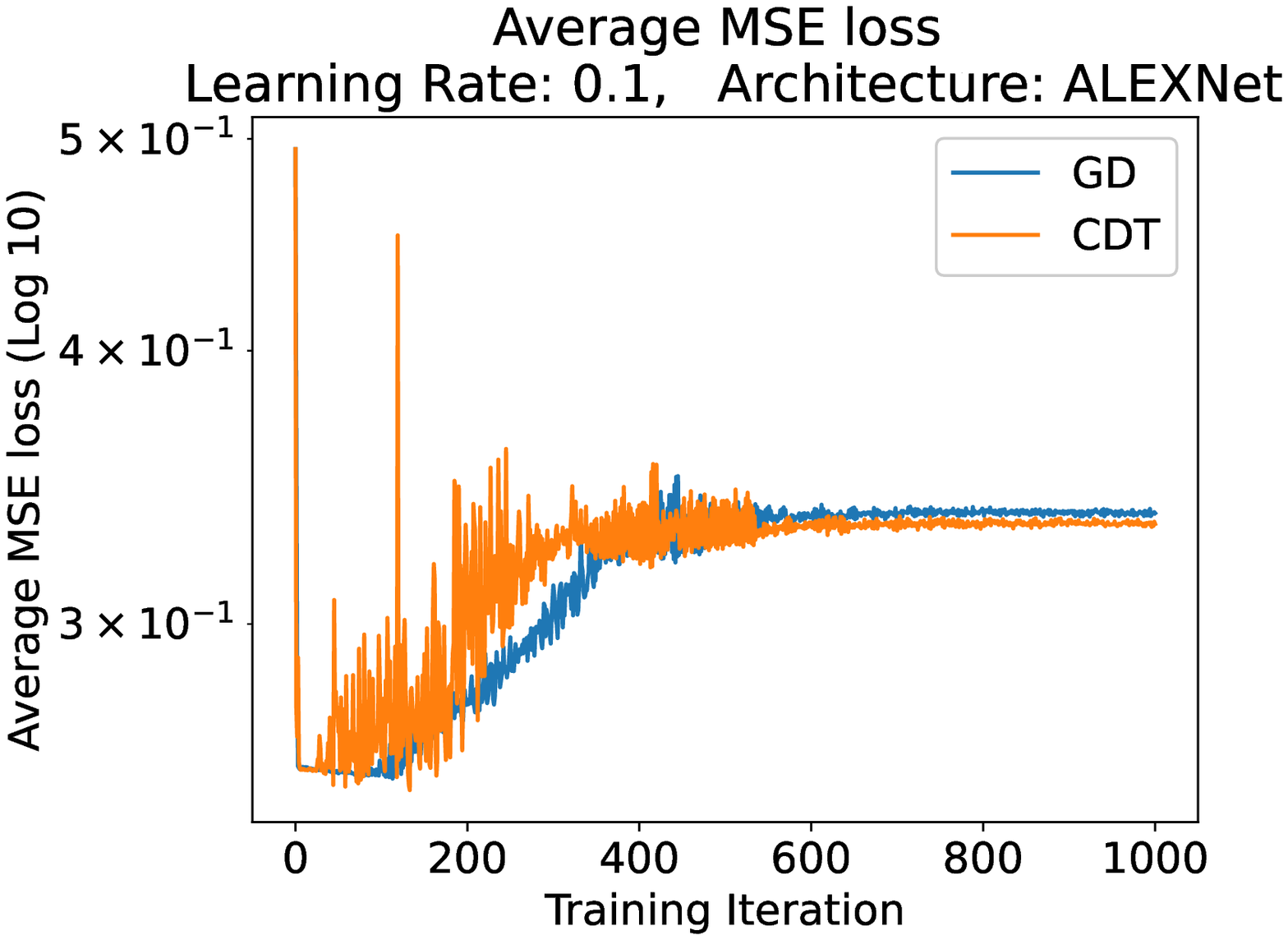}
         \label{fig:alexnet_01}
     \end{subfigure}
      \begin{subfigure}{0.34\textwidth}
         \centering
         \caption{}
         \includegraphics[width=\textwidth, clip]{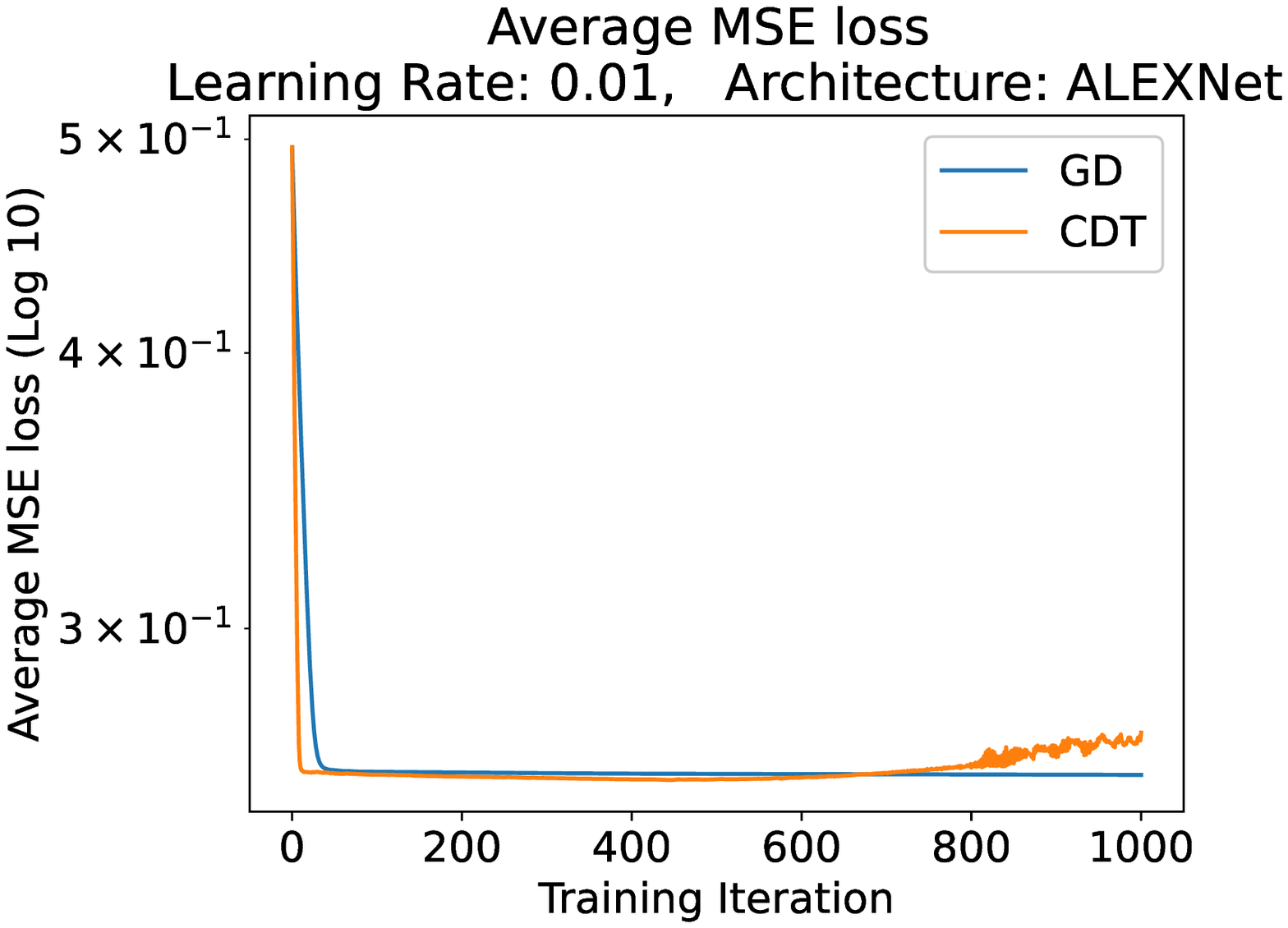}
         \label{fig:alexnet_001}
     \end{subfigure}
     \centering
      \begin{subfigure}{0.34\textwidth}
         \centering
         \caption{}
         \includegraphics[width=\textwidth, clip]{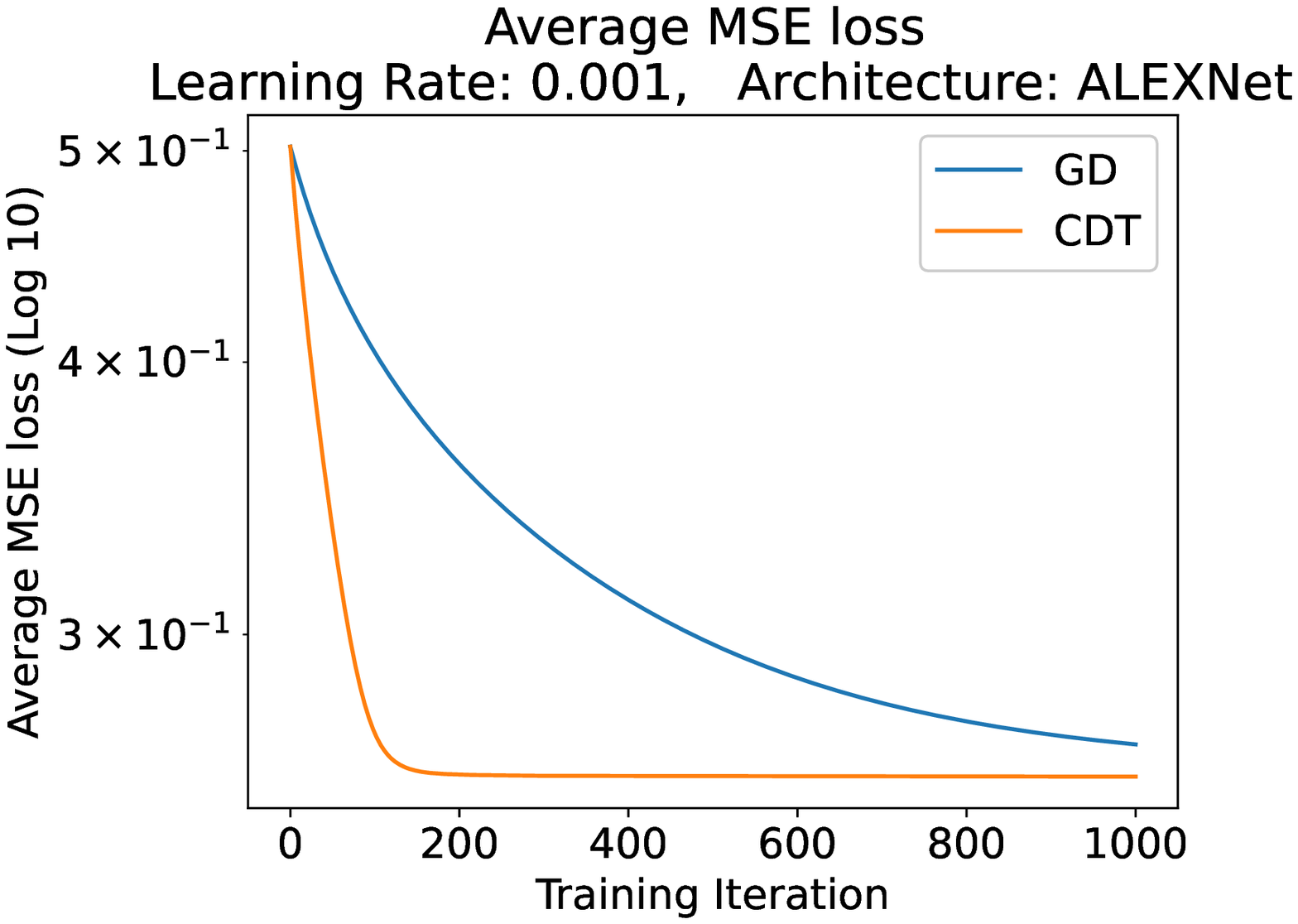}
         \label{fig:alexnet_0001}
     \end{subfigure}
     \caption{Absolute MSE loss on validation data for ALEXNet during training with GD and CDT, averaged over all 10 initializations and data shuffles of classification dataset. As can be seen in b, overfitting occurs for high learning rates $\alpha$. c demonstrate overfitting occuring earlier for CDT due to the acceleration of training.} 
\end{figure}

As can be seen in Table \ref{table:alex}, CDT improves training robustness for CNNs with multiple outputs at higher learning rates. Due to the small training batch size, the performance is poor for both models. ALEXNet is a complex network and will easily overfit the training data. As can be seen in Figure \ref{fig:alexnet_0001}, CDT accelerates learning for CNNs as well as ANNs for low learning rates. However, Figure \ref{fig:alexnet_01} demonstrates that both training algorithms overfit quickly for high learning rates. CDT however stabilizes at a lower loss, indicating higher generalizability after many iterations. More robust experimentation is required to verify this observation. As can be seen in Figure \ref{fig:alexnet_001} CDT converges at a few iterations and reach a lower loss than GD but does however overfit earlier than GD. The observed behavior is expected as CDT accelerates training hence overfit sooner. This hints at using a regularization method together with CDT for optimal performance when using complex network architectures.


\section{Conclusion}
In this paper, a novel model-based control approach to train ANNs under Gradien Decent is proposed. 
The method uses the notion of empirical Neural Tangent Kernels (NTK) of ANN training under gradient descent as a model. After analyzing some baseline properties of the model (solvability, stability), a new fictitious label input is created. Label augments equip the training dynamics with dynamically manipulable and artificial labels. 
These labels give rise to explicit control of the ANNs training behavior. 

The newly developed method of Control Decent Learning hence directly manipulates the label augments whilst being (locally) convergent. In other words, CDT has a locally optimal training behavior via solving a \Htwo optimal control problem.

This novel method is demonstrated to improve loss convergence rate for both known CNN architectures and fully connected ANNs with varying widths and depths. Furthermore, CDT gives local convergence guarantees to target labels increasing robustness of ANN training. The stability analysis of ANN training uncovered the effect choice of learning rate has on local ANN training convergence in the upper bound. Reachability is shown to be a good metric for data learnability from the perspective of the chosen ANN architecture. However, CDT and the reachability analysis demonstrated that due to the accelerated training, overfitting is a larger issue for the novel training method. Therefore CDT should be deployed in conjunction with a regularization method to mitigate this effect.
 

We demonstrated that the theoretical framework of dynamical system theory is directly applicable to ANN training. CDT unlocks the potential to develop additional model-based training solutions. This work is merely the first step in finding a comprehensive description of ANN training suitable for Control Theory applications. We invite the community to further investigate ANN training behavior informed by the NTK from the perspective of dynamical systems and control theory.   

\section{Acknowledgement}
The authors gratefully acknowledge the support of the project OCTON 1, 2 at Chalmers University of Technology. This work was supported in part by the Transport Area of Advance, at Chalmers University of Technology. Moreover, the project was carried out in collaboration with and is supported by Centiro Solutions, a logistics software company based in Sweden.

None of the authors have any conflicts of interest to declare.

\appendix
\section{Existence and uniqueness of solution}
\label{appx:lipschitzness}
The analysis in Section 3 and onward require the ANN training dynamics to have a unique solution on the interval $[k_i,k_f]$. The following proposition is a variation of a proposition on Liptschitzness given in \cite{khalil2002nonlinear}.

\begin{proposition} 
\label{prop:solution}
Suppose that $f_\theta(\hat{y}(k))$ is bounded on the discrete interval $k\in[k_i, k_f]$ and satisfies
\begin{eqnarray}
    &&||f_\theta(\hat{y}_1)-f_\theta(\hat{y}_2)||_2\leq L||\hat{y}_1-\hat{y}_2||_2
\end{eqnarray}
$\forall \hat y_1,\hat y_2 \in \mathbb{R}^{o},\forall k\in [k_i, k_f]$ with $L$ being the Lipschitz constant. Then, for all initial conditions $||f_\theta(\hat{y}_i)||_2\leq \phi$ with a bounded real scalar $\phi$. The discrete difference equation $\hat y (k+1) = f_\theta(\hat{y}(k))$, with $\hat{y}_i=\hat{y}(k_i)$ has a unique solution over the time interval $[k_i,k_f]$.
\end{proposition}
\begin{proof}
By means of the continuity assumption of $\hat y(k)$ in $\theta(k)$, the proof is a direct consequence of \cite{khalil2002nonlinear} (Ch. 2.2, pp 67, Theorem 2.4).
\end{proof}
It follows that when training an ANN under gradient descent a solution to eq. \eqref{eq:outputODE} always exists and that the solution is unique on the time interval $[k_i,k_f]$.

\section{Local training dynamics}
\label{appx:linearization}
The following appendix details the first order Taylor linearization used to derive eq.~\ref{eq:linear}. The following is a variation of the linearization used in \cite{googlepaper}.
 We define the time instance of linearization as $k_0$. In this local aspect, $\hat y(k)$  is described with $\theta(k_0)$ and $\vartheta(k) \equiv \theta(k)-\theta(k_0)$. That is, 
\begin{eqnarray}
\hat y(k)=\hat y(k_0)+\left(\fracpartial{\hat y(k)^T}{\theta(k)}\right)^T_{\theta(k_0)}  \vartheta(k)+\sum_{i=2}^{\infty} \fracpartial{^{(i)}\hat y(k)^T}{^{(i)} \theta(k)}^T_{\theta(k_0)} \frac{\vartheta^{(i)}(k)}{i!} \label{eq:output}
.\end{eqnarray}
Using only the first term from the Taylor expansion (eq. \eqref{eq:output}) results in
\begin{eqnarray}
\label{eq:y_lin}
\hat y_{\vartheta}(k) \equiv \hat y(k_0) +\left(\fracpartial {\hat y(k)^T}  {\theta(k)}\right)^T_{\theta(k_0)}\vartheta(k).
\end{eqnarray}
The smoothness (once continuously differentiable) of the loss function enables the definition of the \emph{local training dynamics} by, 
\begin{eqnarray}
\hat y_{\vartheta}(k+1) = 
 \hat y_{\vartheta}(k)-
\alpha\Theta(k_0) \fracpartial{\mathcal{L}(\hat{y}(k),y)}{\hat{y}(k)}=f_{\vartheta}(\hat y(k)).
\end{eqnarray}

\section{Lagrange error bounds for local training dynamics}
\label{appx:Lagrangian}
In order to quantify the error between the local and the global training dynamics the Lagrange error bound \cite{khalil2002nonlinear} is used,
\begin{eqnarray}
\label{eq:NTK:overbound}
&& ||\hat y(k)-\hat y_{\vartheta} (k)||_2
\leq \max_{\theta(\kappa)}  \frac 12 \left |\left| \left( \fracpartial{^2\hat y(k)^T}{^2\theta(k)}\right)^T_{\theta(\kappa)}  \right| \right |_2 \cdot ||\vartheta(k)||^{2}_2
\end{eqnarray}
where for any $\kappa$ on the discrete interval $[k_0,k]$ $\theta(\kappa)$ is evaluated.
Note that eq. \eqref{eq:NTK:overbound} expresses the overbound of the deviation between the outputs obtained from the local training dynamics $\Theta(k_0)$ and the global training dynamics $\Theta(k)$. Meaning, while the linear dynamics are not replicating the learning behavior exactly we can still quantify the goodness of the approximation. 

\section{Loss Linearization}
\label{appx:losslin}
Although, the local training dynamics are linearized w.r.t.~$\theta(k_0)$, the derivative of the loss $\left(\fracpartial{\mathcal{L}(\hat{y}(k),y)}{ \hat{y}(k)}\right)_{\theta(k_0)}$ can still be a nonlinear function of the output $\hat{y}(k)$ (e.g., for cross entropy loss). When this is the case we propose a further linearization step and apply a first-order Taylor series approximation on the loss derivative:
\begin{eqnarray}
\left(\fracpartial{\mathcal{L}(\hat{y}(k),y)}{ \hat{y}(k)}\right)_{\theta(k_0)}  = & \left(\fracpartial{\mathcal{L}(\hat{y}(k_0),y)}{ \hat{y}(k_0)}\right)_{\theta(k_0)}  + \left(\fracpartial{^2\mathcal{L}(\hat{y}(k_0),y)}{^2\hat{y}(k_0)}\right)_{\theta(k_0)}(\hat y(k)-\hat y(k_0))   \nonumber \\ & +  \sum_{i=2}^{\infty} \left(\fracpartial{^{(i+1)}\mathcal{L}(\hat{y}(k_0),y)}{^{(i+1)}\hat{y}(k_0)}\right)_{\theta(k_0)}\frac{(\hat y(k)-\hat y(k_0))^{(i)}}{i!} 
\end{eqnarray}
and 
\begin{equation}
    \left(\fracpartial{\mathcal{L}(\hat{y}(k),y)}{ \hat{y}(k)}\right)_{L} =  \left(\fracpartial{\mathcal{L}(\hat{y}(k_0),y)}{ \hat{y}(k_0)}\right)_{\theta(k_0)}  + \left(\fracpartial{^2\mathcal{L}(\hat{y}(k_0),y)}{^2\hat{y}(k_0)}\right)_{\theta(k_0)}(\hat y(k)-\hat y(k_0))
\end{equation}
with Lagrange error bound
\begin{eqnarray}
    \label{eq:LossLagrangeError}
&& \left |\left|\left(\fracpartial{\mathcal{L}(\hat{y}(k),y)}{ \hat{y}(k)}\right)_{\theta(k_0)}-\left(\fracpartial{\mathcal{L}(\hat{y}(k),y)}{ \hat{y}(k)}\right)_{L} \right| \right |_2
  \nonumber \\ && \leq  \frac 12 \left |\left| \left(\fracpartial{^3\mathcal{L}(\hat{y}(k_0),y)}{^3\hat{y}(k_0)}\right)_{\theta(k_0)}  \right| \right |_2 \cdot ||(\hat y(k)-\hat y(k_0))||_2^2.
\end{eqnarray}
Next, insert the linearized loss into eq.~\eqref{eq:linear} and assume $\hat y(k) \approx \hat y_\vartheta(k)$ and $\hat y(k_0) = \hat y_\vartheta(k_0)$. We define the \textit{control oriented training dynamics} as
\begin{eqnarray}
\label{eq:lin_trainig_dynamics}
  &  \hat y_{\vartheta}(k+1) = 
\nonumber \\ & \hat y_{\vartheta}(k) -  
\alpha\Theta(k_0)\left( \left(\fracpartial{\mathcal{L}(\hat{y}(k_0),y)}{ \hat{y}(k_0)}\right)_{\theta(k_0)} + \left(\fracpartial{^2\mathcal{L}(\hat{y}(k_0),y)}{^2\hat{y}(k_0)}\right)_{\theta(k_0)}(\hat y_\vartheta(k)-\hat y(k_0)) \right) = \nonumber \\ &
\left( I - \alpha\Theta(k_0)\left(\fracpartial{^2\mathcal{L}(\hat{y}(k_0),y)}{^2\hat{y}(k_0)}\right)_{\theta(k_0)} \right)\hat y_{\vartheta}(k) \nonumber \\ & + \alpha\Theta(k_0)\left(\fracpartial{^2\mathcal{L}(\hat{y}(k_0),y)}{^2\hat{y}(k_0)}\right)_{\theta(k_0)} \hat y(k_0) - \alpha\Theta(k_0) \left(\fracpartial{\mathcal{L}(\hat{y}(k_0),y)}{ \hat{y}(k_0)}\right)_{\theta(k_0)}.
\end{eqnarray}
Note that there is a bias term $\alpha\Theta(k_0)\left(\fracpartial{^2\mathcal{L}(\hat{y}(k_0),y)}{^2\hat{y}(k_0)}\right)_{\theta(k_0)} \hat y(k_0) - \alpha\Theta(k_0) \left(\fracpartial{\mathcal{L}(\hat{y}(k_0),y)}{ \hat{y}(k_0)}\right)_{\theta(k_0)}$ that only offsets the dynamics. 
The cumulative error bound (based on eq. ~\eqref{eq:NTK:overbound} and eq. ~\eqref{eq:LossLagrangeError} can be given as follows. Denote the left hand side of eq.~\eqref{eq:NTK:overbound} with $E_y$ and eq. ~\eqref{eq:LossLagrangeError} with $E_L$. Consider eq. ~\eqref{eq:LossLagrangeError} and inject the linearization error of $\hat y(k)$ as 
\begin{equation}
    E_L(E_y) \leq \frac 12 \left |\left| \left(\fracpartial{^3\mathcal{L}(\hat{y}(k_0),y)}{^3\hat{y}(k_0)}\right)_{\theta(k_0)}  \right| \right |_2 \cdot ||(\hat y_\vartheta(k) + E_y -\hat y(k_0))||_2^2.
\end{equation}
Then, the error bound for the control-oriented training dynamics can be given by  
\begin{equation}
    E \leq E_y - \alpha\Theta(k_0) E_L(E_y).
\end{equation}
 
In certain cases, when the loss is a quadratic function of the output (e.g., SSE, or MSE losses), the linearization error of the loss disappears. 

\section{Examples of equilibrium points}
\label{appx:equilibriums}
The following appendix discusses the conditions under which an equilibrium point may exist. The definition of an equilibrium point is a point where no change to the ANN output occurs, i.e  $\hat y (k + 1) = \hat y (k)$. In case of the global training dynamics $\hat y (k + 1) = \hat y (k)$ can only occur if $\alpha\Theta(k) \left( \fracpartial{\mathcal{L}(\hat{y}(k),y)}{ \hat{y}(k)}\right)_{\theta(k)} = 0$. More precisely, there is an equilibrium point if any of the following conditions are fulfilled.
\begin{enumerate}
    \item The most important case is when the loss is at a (local) minimum, $\left( \fracpartial{\mathcal{L}(\hat{y}(k),y)}{ \hat{y}(k)}\right)_{\theta(k)} = 0$.
    \item The learning is frozen $\alpha = 0$.
    \item The kernel is a null matrix $\Theta(k) = \underline{\underline{0}}$. However, it can only occur in some very specific cases, e.g., if $\fracpartial{ \hat{y}(k, x_i)^T}{\theta(k)}$ and  $\fracpartial{ \hat{y}(k, x_j)^T}{\theta(k)}$ for all data combinations $x_i$, $x_j$. 
    \item A less trivial case is when $\Theta(k) \left( \fracpartial{\mathcal{L}(\hat{y}(k),y)}{ \hat{y}(k)}\right)_{\theta(k)}$ is a zero vector while $\Theta(k) \neq \underline{\underline{0}}$, and $\left( \fracpartial{\mathcal{L}(\hat{y}(k),y)}{ \hat{y}(k)}\right)_{\theta(k)} \neq 0$. I.e., the derivative of the loss $\left( \fracpartial{\mathcal{L}(\hat{y}(k),y)}{ \hat{y}(k)}\right)_{\theta(k)}$ is in the null space of the kernel. 
\end{enumerate}

\section{Boundedness for common losses}
\label{appx:boundednessLosses}
The boundedness of some common loss functions is analyzed, assuming static target $y$. 
\begin{itemize}
    \item \textbf{Mean squared error (MSE) loss}: The MSE loss is given as $\mathcal{L}(\hat{y}(k), y) = \frac{1}{2rn_L}(\hat y_{\vartheta}(k) - y)^2$. Substituting the MSE loss in eq.~\eqref{eq:linear} one gets
    \begin{equation}
        \hat y_{\vartheta}(k+1) = \hat y_{\vartheta}(k)-\frac{\alpha\Theta(k_0)}{rn_L}(\hat y_{\vartheta}(k) - y).
        \label{eq:mseode}
    \end{equation}
    This difference equation has an equilibrium point at a bounded $y$, which is proven in later sections. 
    For the linear time-discrete ANN training dynamics under MSE loss
    \begin{equation}
    \label{eq:mseode2}
        \hat y_{\vartheta}(k+1)=\left(I-\frac{\alpha\Theta(k_0)}{rn_L}\right)\hat y_{\vartheta}(k) + \frac{\alpha\Theta(k_0)}{rn_L}y.
    \end{equation}
    In eq.~\eqref{eq:mseode2}, trajectories of $ \hat y_{\vartheta}(k)$ can be checked for boundedness by looking at the eigenvalues of the system matrix $\left(I-\frac{\alpha\Theta(k_{0})}{rn_L}\right)$.
    The local training dynamics are internally exponentially bounded iff
    \begin{equation}
    |\lambda| < 1 \quad \forall \lambda \in eig\left(I-\frac{\alpha\Theta(k_0)}{rn_L}\right).
    \end{equation}
    The proof for this can be found in \cite{rugh}. 
    \item \textbf{Sum of Squared Error (SSE) loss}: The SSE loss is similar to the MSE loss without the normalization with $rn_L$, i.e., $\mathcal{L}(\hat y_{\vartheta}(k), y) = \frac{1}{2}(\hat y_{\vartheta}(k) - y)^2$. Therefore, following the same line of thought as for the MSE, if 
    \begin{equation}
    |\lambda| < 1 \quad \forall \lambda \in eig(I-\alpha\Theta(k_0)),
    \end{equation}
    then $\hat y_{\vartheta}(k)$ does not diverge from $y$.
    Since $rn_L$ is a positive integer, the overbound for a non-divergent $\alpha$ with SSE loss is smaller than with MSE loss. 
    \item \textbf{Mean absolute error (MAE) loss}: The mean absolute error loss is given as $\mathcal{L}(\hat y_{\vartheta}(k), y) = \frac{1}{rn_L}||\hat y_{\vartheta}(k) - y||_1$ and its derivative w.r.t. $\hat y_{\vartheta}(k)$ is 
    \begin{equation}
        \fracpartial{\mathcal{L}(\hat y_{\vartheta}(k),y_i)}{\hat y_{\vartheta}(k)} = \frac{1}{rn_L} \sum_{i=1}^{rn_L} \frac{\hat y_{\vartheta,i}(k) - y}{|\hat y_{\vartheta,i}(k) - y_i|}
    \end{equation}
    for $\hat y_{\vartheta,i}(k) \neq y_i \forall i$. Index $i$ denotes one element of the vector-valued outputs and labels. Outside of $\hat y_{\vartheta,i}(k) = y_i$, the derivative is $-1$ if $\hat y_{\vartheta,i}(k) < y_i$ and $1$ if $\hat y_{\vartheta,i}(k) > y_i$. Therefore, the discrete learning dynamics with MAE loss can be written as
    \begin{equation}
        y_{\vartheta}(k+1)= 
y_{\vartheta}(k) - \frac{\alpha \Theta(k_0)}{rn_L}sgn(\hat y_{\vartheta}(k) - y).
    \end{equation}
    Intuitively, this means the loss will uniformly converge to the $\frac{\alpha \Theta(k_0)}{rn_L}$ radius of $y$. The conditions for exponential internal boundedness are not fulfilled.
    \item \textbf{Cross entropy loss}: The cross entropy loss or log loss is used for classification, rather than regression tasks. It can be computed as $\mathcal{L}(\hat y_{\vartheta}(k), y) = -y^Tlog(\hat y_{\vartheta}(k))$. Then, the nonlinear difference-equation for the learning dynamics is
    \begin{equation}
    \label{eq:CEloss}
        y_{\vartheta}(k+1)= y_{\vartheta}(k) + \alpha \Theta(k_0) \check y_{\vartheta}(k)y,
    \end{equation}
    where $\check y_{\vartheta}(k)$ is a diagonal matrix $\in \mathbb{R}^{rn_L \times rn_L}$ of the element-wise inverses of $\hat y_{\vartheta}(k)$, assuming $\hat y_{\vartheta}(k)$ has no zero elements. Then, $y$ is an equilibrium point for the difference equation if $\alpha \Theta(k_0)log(\check y_{\vartheta}(k)y$ is a null vector. I.e.,~$y$ is an equilibrium point if it is in the nullspace of the matrix $\alpha \Theta(k_0)\check y_{\vartheta}(k)$. A trivial solution to this if $y=0$, and this is the only solution if the columns in $\alpha \Theta(k_0)\check y_{\vartheta}(k)$ are linearly independent. If they are linearly dependent, there are infinitely many equilibrium points. For more in-depth analysis a Lyapunov function is sought to give boundedness conditions for the cross entropy loss. In discrete-time, Lyapunov boundedness is fulfilled if $V(f(x)) - V(x) < 0$, where $V(x)$ is a Lyapunov function \cite{khalil2002nonlinear}.
    Let $V(x) = x^Tx$ be a Lyapunov function. Then for eq.~\eqref{eq:CEloss} the Lyapunov boundedness criteria is
    \begin{equation}
        \left( y_{\vartheta}(k) + \alpha \Theta(k_0) \check y_{\vartheta}(k)y \right)^T\left( y_{\vartheta}(k) + \alpha \Theta(k_0) \check y_{\vartheta}(k)y \right) - y_{\vartheta}(k)^Ty_{\vartheta}(k) <0
    \end{equation}
    which can be simplified to
    \begin{equation}
        \alpha^2 y^T \check y_{\vartheta}^T(k) \Theta^T(k_0) \Theta(k_0)  \check y_{\vartheta}(k) y - 2 \alpha y_{\vartheta}^T(k)\Theta(k_0)  \check y_{\vartheta}(k) y < 0.
    \end{equation}
    Although, this equation is easy to check whether it is fulfilled or not, a universal conclusion cannot be drawn for the global boundedness. On the other hand, the cross entropy loss is mainly used for classification tasks rather than regression where the target $y$ and the output $y_{\vartheta}(k)$ are normalized, i.e.,  $y, \; y_{\vartheta}(k) \in (0,1) \subset \mathbb{R}^{rn_L}$. In such a case (in a local sense) $\alpha$ is always positive, $y$, and $y_{\vartheta}(k)$ are positive vectors, $\check y_{\vartheta}(k)$ is a diagonal matrix with positive elements. $\Theta(k_0)$ is symmetric and if the input is normalized, it is positive-definite too \cite{jacot}. Then, for a sufficiently small $\alpha$, Lyapunov boundedness is fulfilled.
\end{itemize}
From the above list it is obvious, that from a control-oriented perspective, the SSE and MSE losses are the most appropriate.

\section{Boundedness of the global dynamics}
\label{appx:boundedness}
A criteria for the boundedness of the global training dynamics can be given based on the linearized dynamics. To this end, we subtract the Lagrange error (eq.~\eqref{eq:NTK:overbound}) from eq.~\eqref{eq:LinearizedExpBoundedness} giving a less conservative bound for the global training dynamics:
\begin{equation}
    ||\hat y(k)-\hat y_e||_2\leq \gamma e^{-\lambda(k-k_0)} ||\hat y_{\vartheta}(k_0) - \hat y_e||_2 - \max_{\theta(r)}  \frac 12 \left |\left| \left( \fracpartial{^2\hat y(k)^T}{^2\theta(k)}\right)^T_{\theta(r)}  \right| \right |_2 \cdot ||\vartheta^{2}(k)||_2.
\end{equation}
The above expression suggests that the global training dynamics is not exponentially bounded given the Lagrange error is nonzero. On the other hand, it has important implications on the validity of the linearized training dynamics.
Since the linearization error grows over time, a time instant $k_c > k_0$  can be found where exponential boundedness for the local training dynamics gets violated. I.e., if 
\begin{eqnarray}
    ||\hat y_\vartheta(k_c)-\hat y_e||_2 > && \gamma e^{-\lambda(k_c-k_0)} ||\hat y_{\vartheta}(k_0) - \hat y_e||_2 \nonumber \\ - &&\max_{\theta(r)}  \frac 12 \left |\left| \left( \fracpartial{^2\hat y(k_c)^T}{^2\theta(k_c)}\right)^T_{\theta(r)}  \right| \right |_2 \cdot ||\vartheta^{2}(k_c)||_2
\end{eqnarray}
we can explicitly say that the linear model is poor and must be recalculated. 

\section{The DARE equation}
The following appendix describes the standard Discrete-time Algebraic Riccati Equation (DARE) \cite{kwakernaak1972linear} which is deployed in order to  find the solution $P$ to the quadratic infinite optimization problem in eq.~\eqref{eq:LQC}. The proof that eq.~\eqref{eq:riccati} solves the cost given in eq.~\eqref{eq:LQC} is given in \cite{kwakernaak1972linear}.

\label{appx:DARE}
\begin{eqnarray}
A= \begin{bmatrix}     I-\alpha\Theta(k_0) & \alpha\Theta(k_0) y \\
     \mathbf{0} & 1
\end{bmatrix}, \,\,\,
B=\begin{bmatrix}
     \alpha\Theta(k_0)\\ \textbf{0}
 \end{bmatrix}, \,\,\,
P = A^TPA+Q-A^TPB\left (R+B^TPB \right)^{-1}B^TPA
\label{eq:riccati}
\end{eqnarray}

\section{Initialization of the ANN}
\label{appx:init}
There are three common ways to initialize neural networks of infinite width to derive fixed kernels.
\begin{itemize}
    \item Standard initialization. The weight for each neuron are given as $\mathcal{N}(0, \frac{\sigma_w^2}{s n_l})$ ($\mathcal{N}(0, \frac{\sigma_w^2}{s n_l n_m})$ for convolutional layers), and biases are $\mathcal{N}(0, \sigma_b^2)$ with $\sigma_w$, and $\sigma_b$ being initialization variances, $n_l$ is the width of each layer, $n_m$ is the number of spatial positions in the convolution kernel, and $s$ is a width-scaling factor that goes to $\infty$ for infinite width networks. The main issue with this initialization is that in the infinie width-limit the entries of the NTK diverge.
    \item NTK initialization, proposed by \cite{jacot}. In this case, weights and biases are initialized with normalized gaussian distributions $\mathcal{N}(0,1)$. The weights are multiplied with $\frac{\sigma_w}{\sqrt{s n_l}}$, ($\frac{\sigma_w}{\sqrt{s n_l n_m}}$ for convolutional layers), and the biases are scaled with $\sigma_b$. That is to make the NTK values converge. 
    \item Improved standard initialization \cite{sohl2020infinite}. The difference between the standard and the improved version is that the width-scaling factor is pulled out from the normal distribution, i.e.~$\frac{1}{\sqrt{s}}\mathcal{N}(0, \frac{\sigma_w^2}{n_l})$ and $\frac{1}{\sqrt{s}}\mathcal{N}(0, \frac{\sigma_w^2}{n_l n_m})$. 
    \end{itemize}

The initializations are summarized in Table \ref{tab:NTK_init}. 
    
According to \cite{park2019effect, sohl2020infinite}, infinite width networks with various architectures achieve similar error regardless of initialization. I.e., if they converge, the final value will be similar in output space, regardless of initialization. On the other hand, it is not the case in parameter space; the NTK will take different final numerical values depending on initialization. This means it will traverse a different trajectory during learning since the eigenvalues of the NTK will influence the learning dynamics.

All experiments, both regression and classification, implement initialization 2 as recommended by \cite{jacot}.
\newpage
\section{Supplementary figures}
\label{appx:figs}
\subsection{Regression experiment supplements}
\begin{figure}[!htp]
    \centering
         \begin{subfigure}{0.34\textwidth}
         \centering
         \caption{}
         \includegraphics[width=\textwidth]{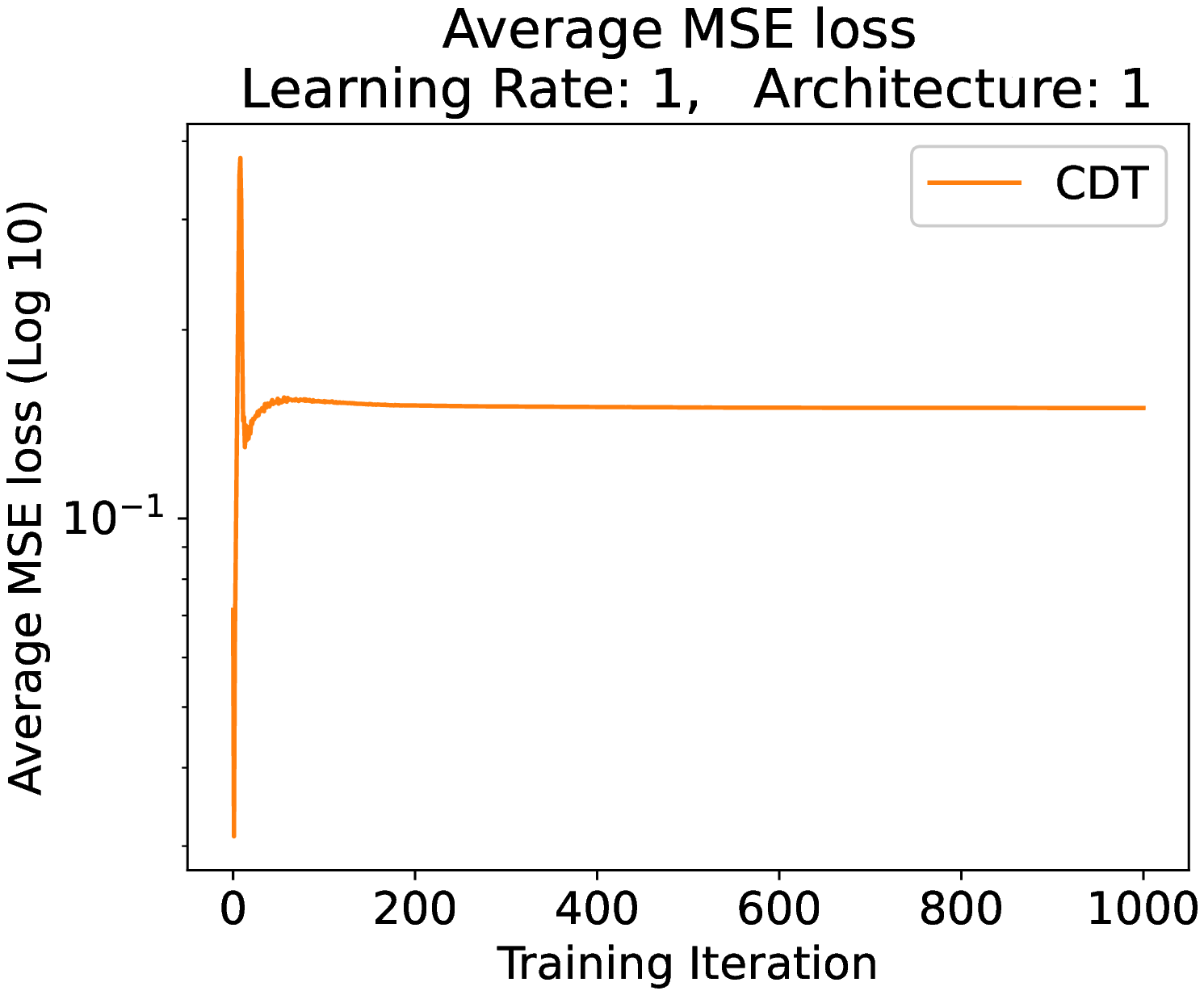}
         \label{fig:a0_1}
     \end{subfigure}
     \begin{subfigure}{0.34\textwidth}
         \centering
         \caption{}
         \includegraphics[width=\textwidth]{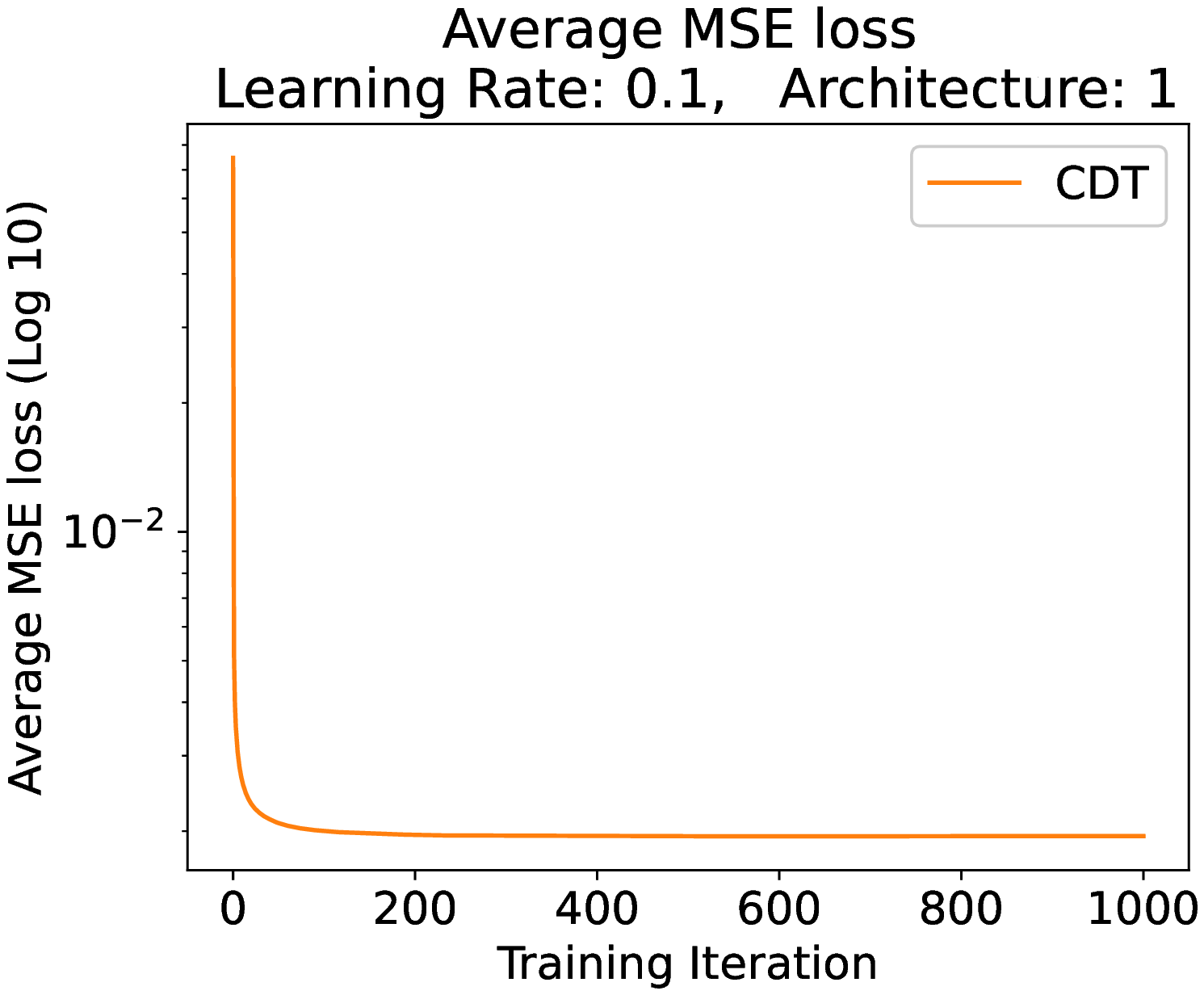}
         \label{fig:a0_01}
     \end{subfigure}
      \begin{subfigure}{0.34\textwidth}
         \centering
         \caption{}
         \includegraphics[width=\textwidth]{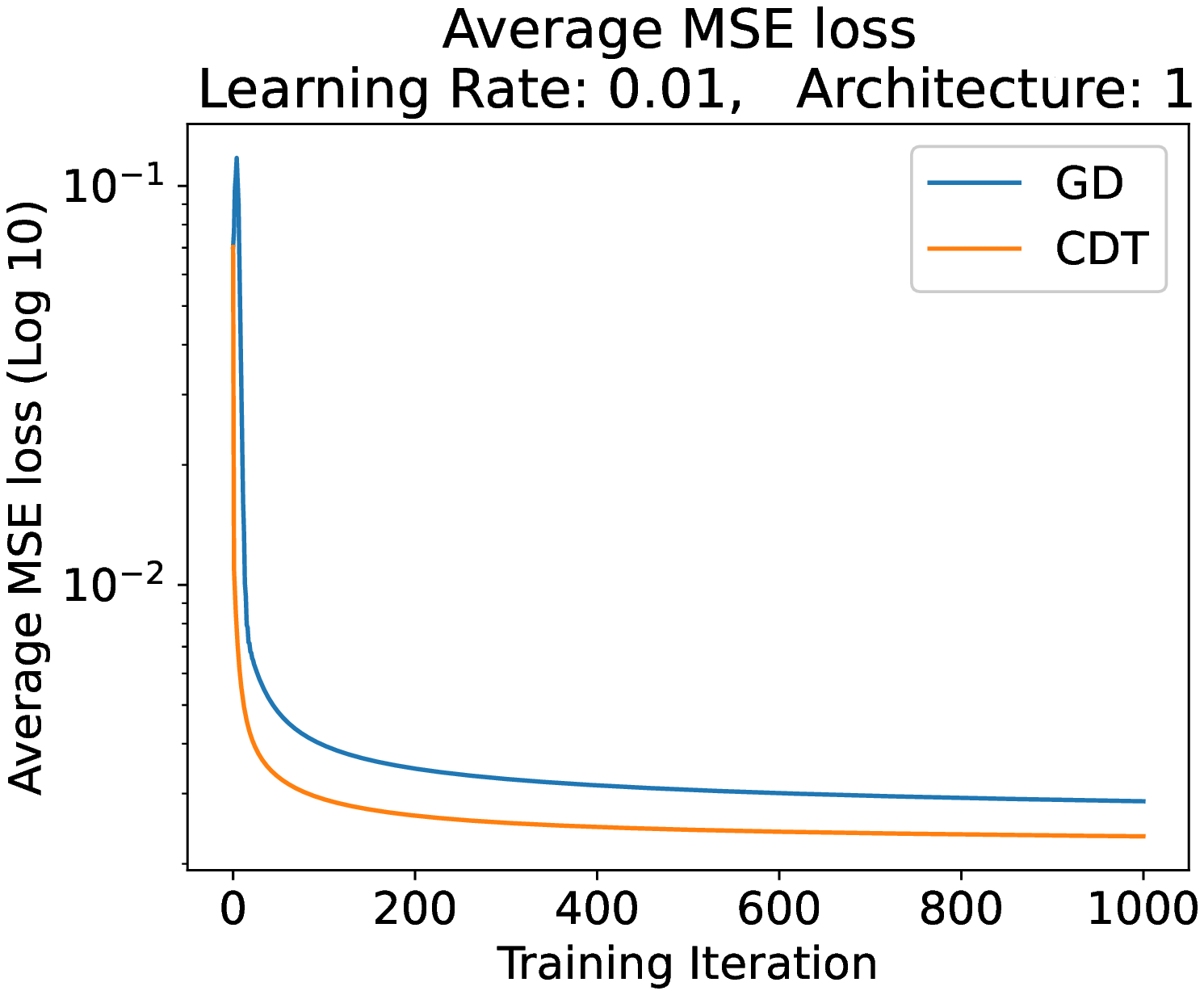}
         \label{fig:a0_001}
     \end{subfigure}
      \begin{subfigure}{0.34\textwidth}
         \centering
         \caption{}
         \includegraphics[width=\textwidth]{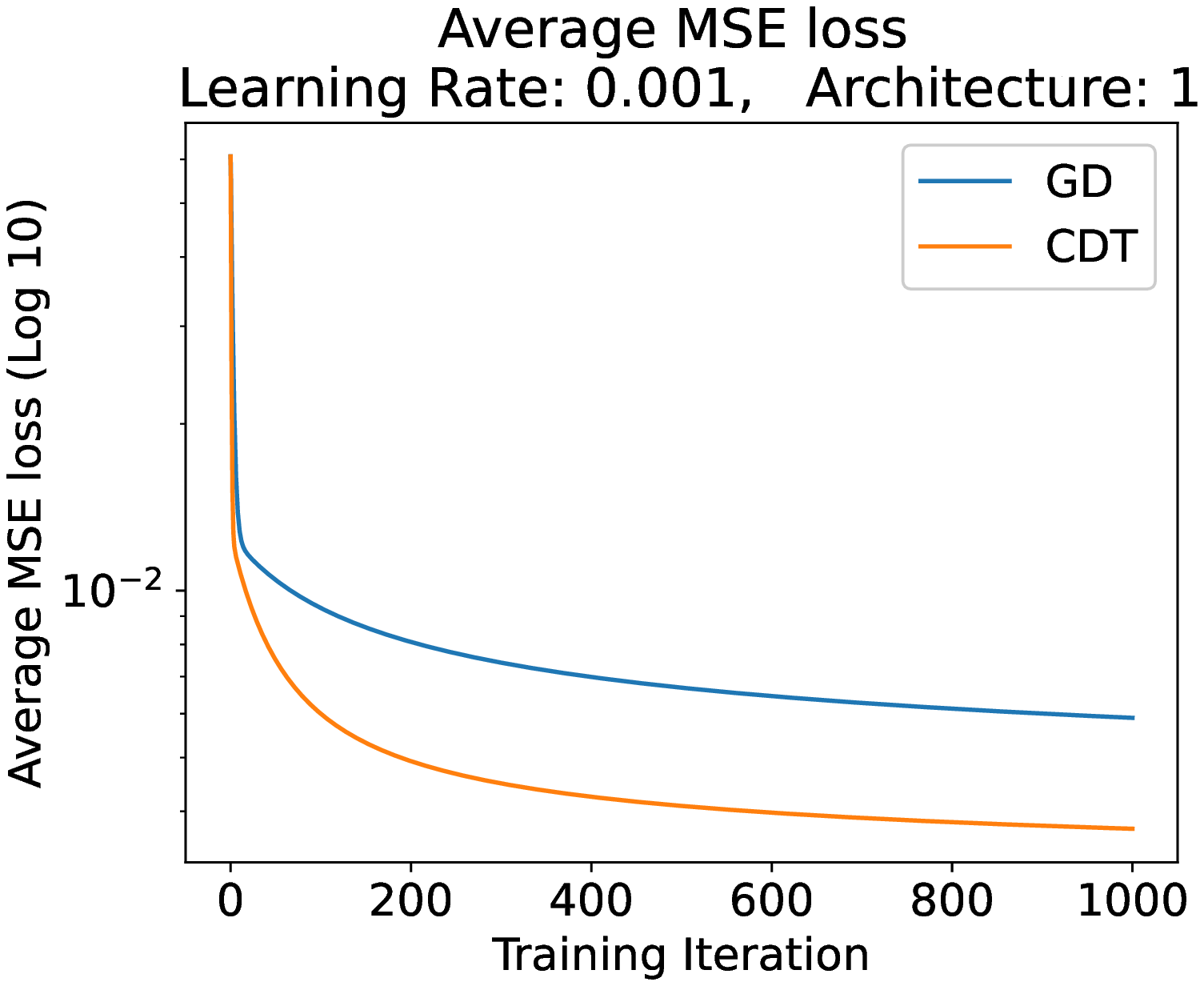}
         \label{fig:a0_0001}
     \end{subfigure}
     \caption{Absolute MSE loss on validation data for fully connected ANN architecture 1 averaged over all  10 initializations and data shuffles during training with GD and CDT. CDT  converges with fewer iterations than GD for all learning rates. Moreover, CDT is stable at higher learning rates ($\alpha\geq0.1$), while GD diverges to $\infty_+$. While reaching a finite steady-state behavior for learning rate $\alpha = 1$ CDT is not observed to converge to the true labels.}
\end{figure}
\newpage
\begin{figure}[!htp]
      \begin{subfigure}{0.34\textwidth}
         \centering
         \caption{}
         \includegraphics[width=\textwidth]{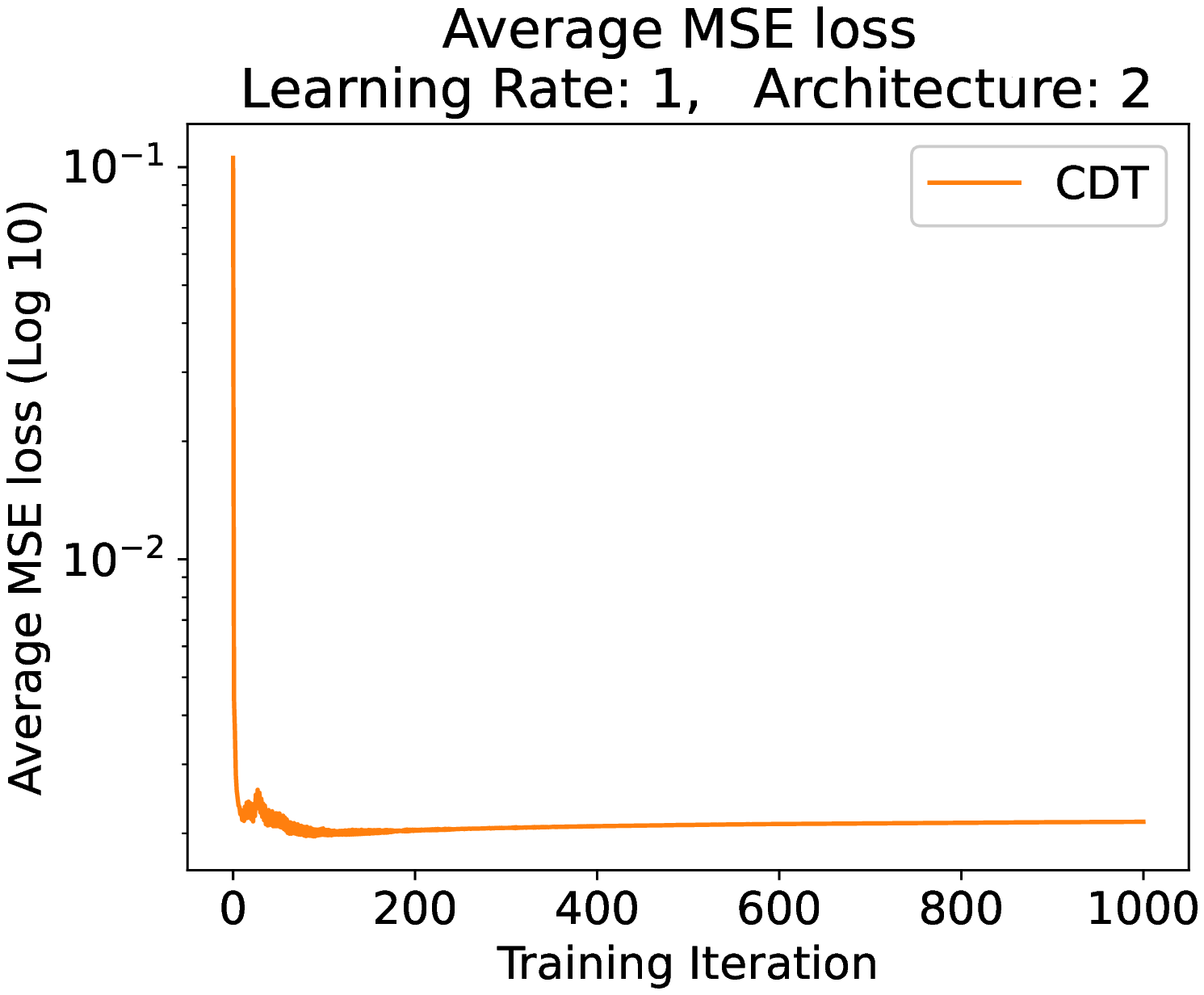}
         \label{fig:a1_01}
     \end{subfigure}
      \begin{subfigure}{0.34\textwidth}
         \centering
         \caption{}
         \includegraphics[width=\textwidth]{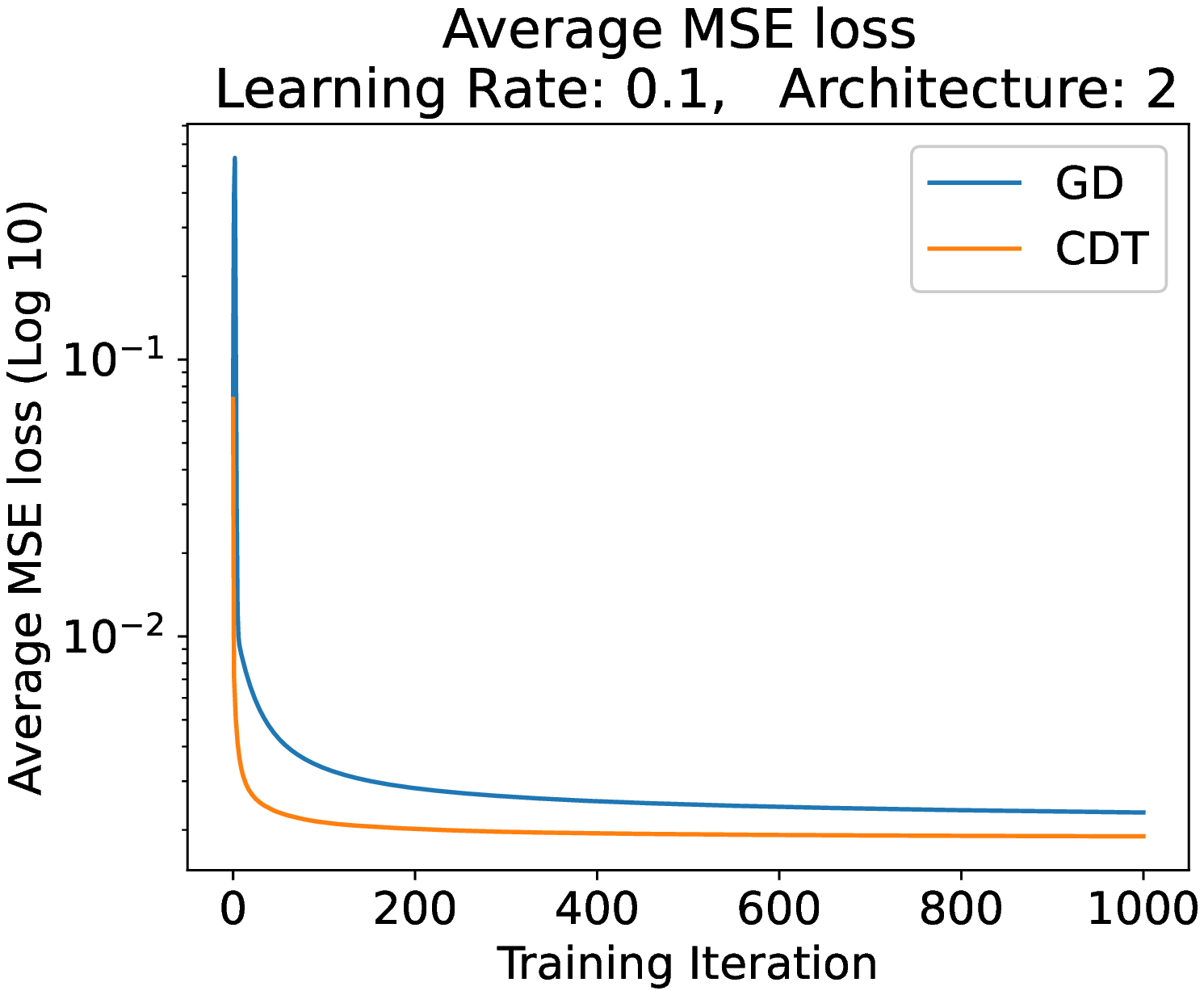}
         \label{fig:a1_01}
     \end{subfigure}
       \begin{subfigure}{0.34\textwidth}
         \centering
         \caption{}
         \includegraphics[width=\textwidth]{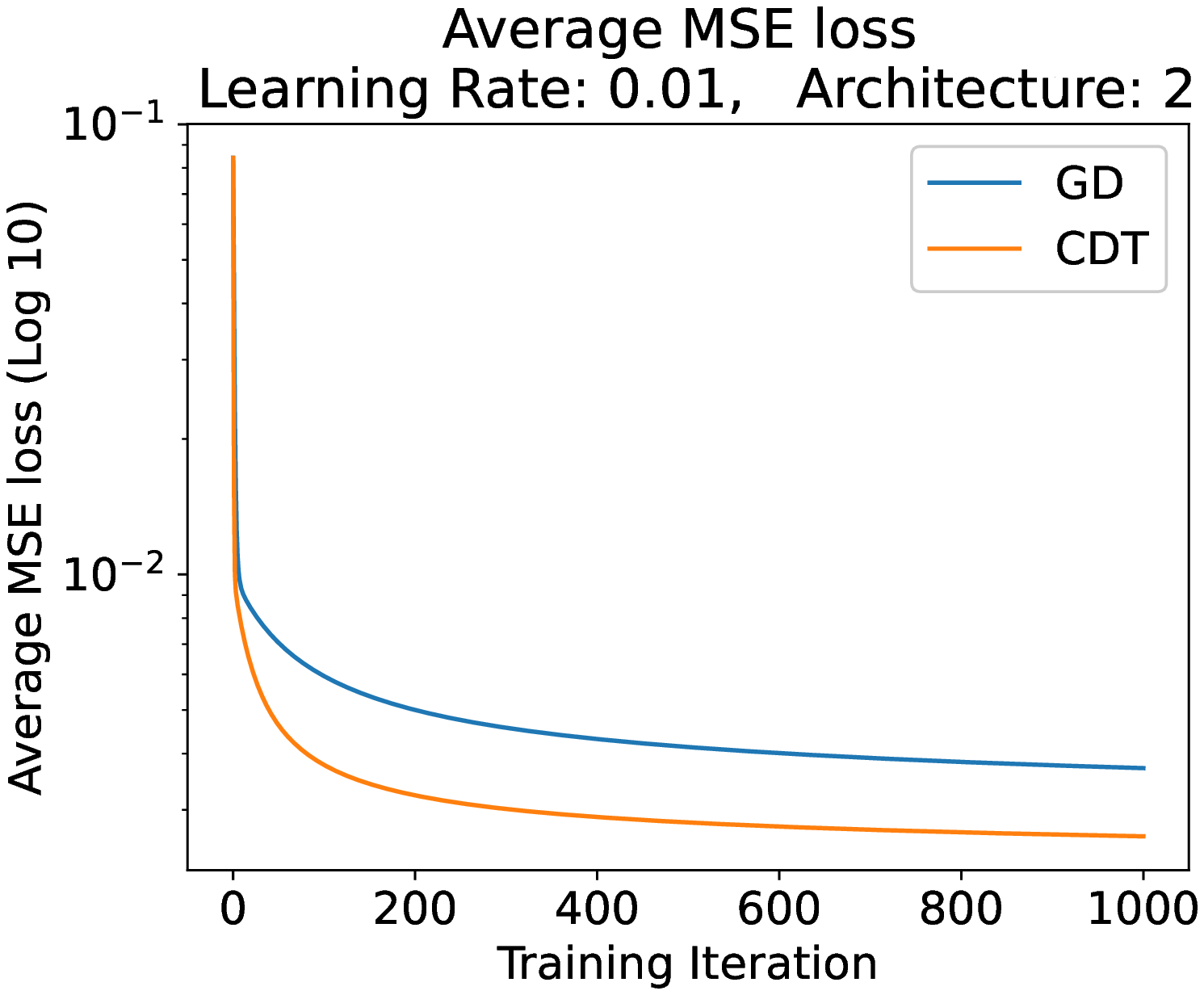}
         \label{fig:a1_001}
     \end{subfigure}
     \centering
           \begin{subfigure}{0.34\textwidth}
         \centering
         \caption{}
         \includegraphics[width=\textwidth]{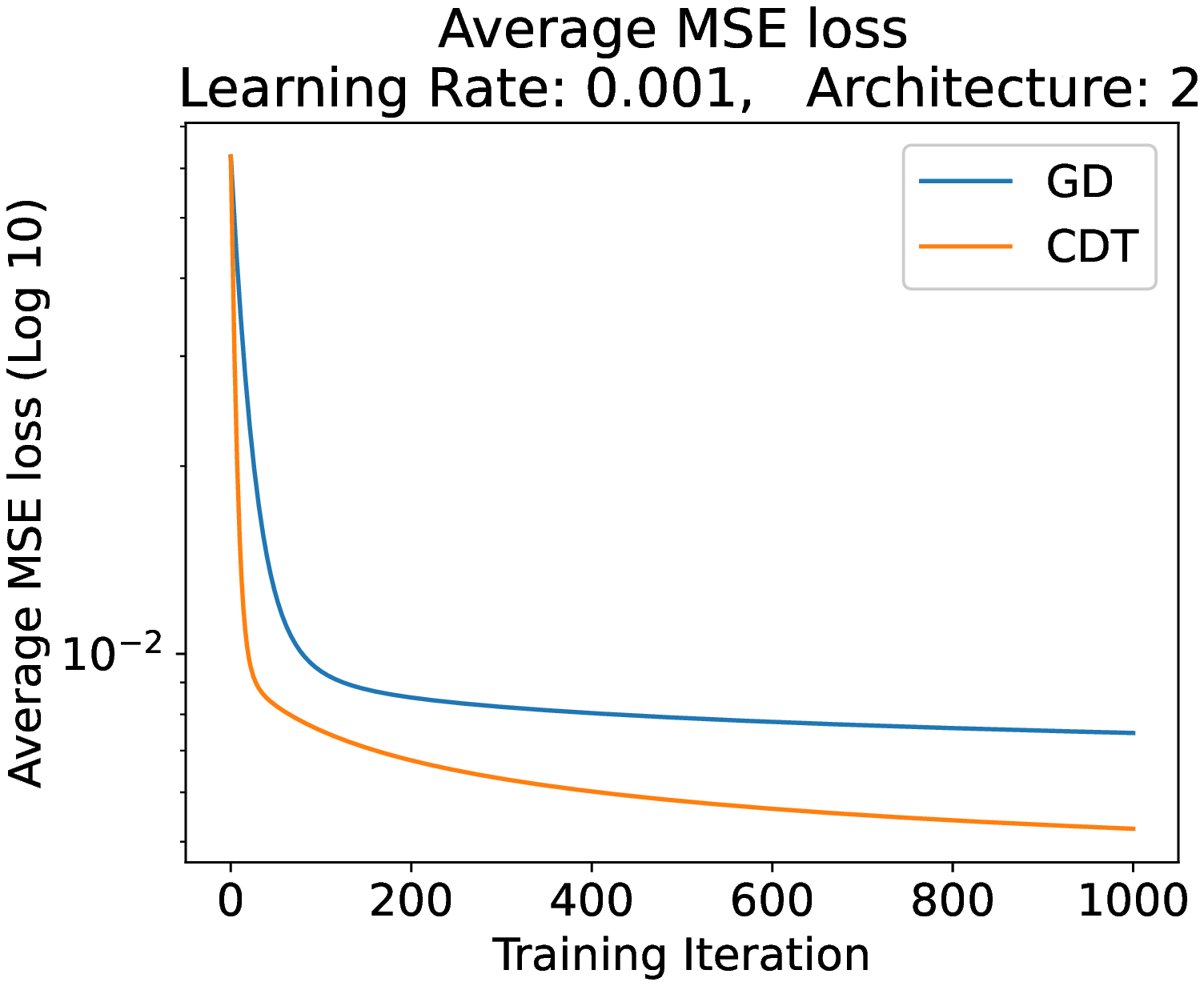}
         \label{fig:a1_0001}
     \end{subfigure}     
     \caption{Absolute MSE loss on validation data for fully connected ANN architecture 2 averaged over all  10 initializations and data shuffles during training with GD and CDT. CDT appears to converge with fewer iterations than GD for all learning rates. Moreover, CDT is stable at higher learning rates ($\alpha\geq1$), while GD diverges to $\infty_+$.}
\end{figure}
 \begin{figure}[!htp] 
     \begin{subfigure}{0.34\textwidth}
         \centering
         \caption{}
         \includegraphics[width=\textwidth]{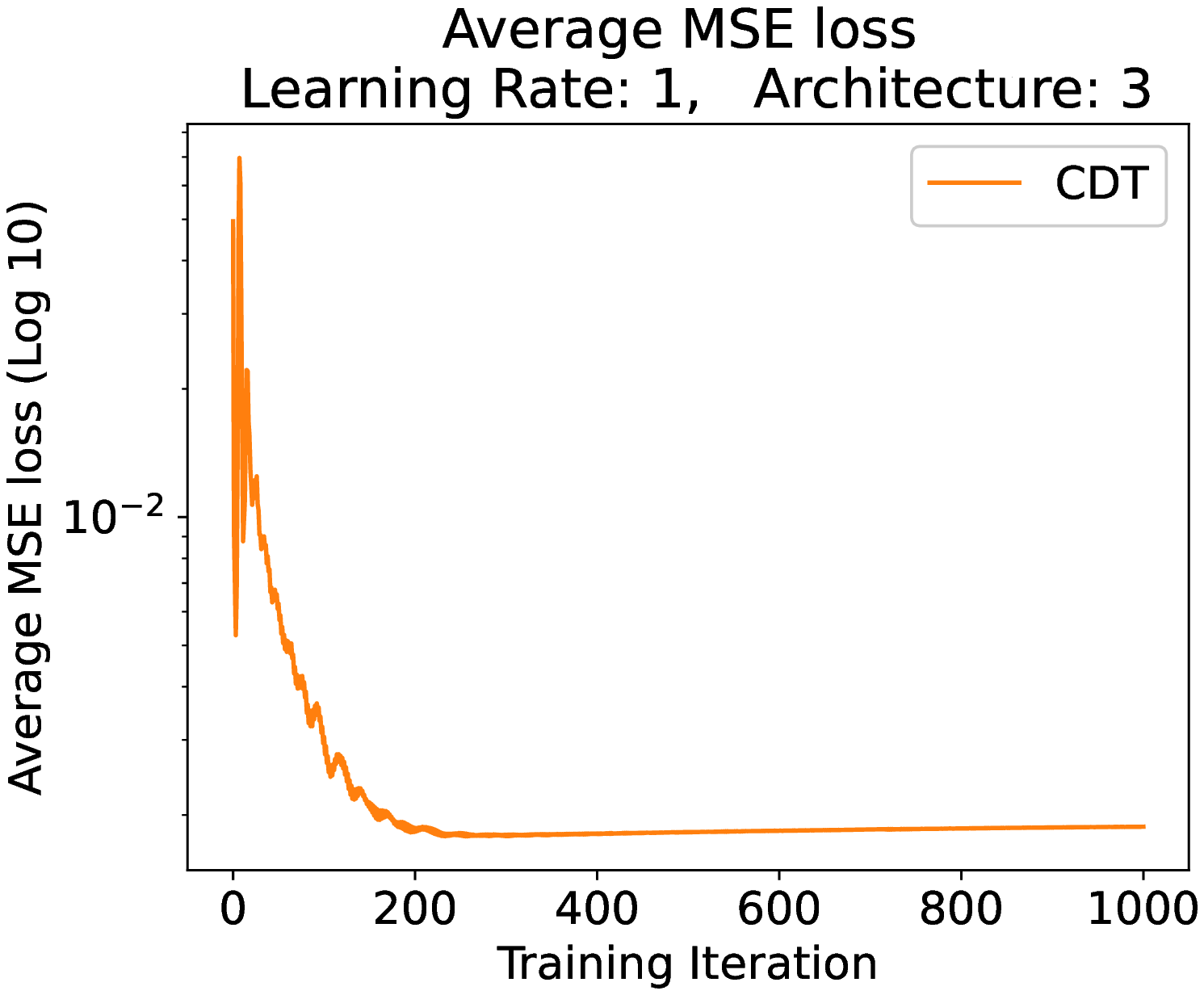}
         \label{fig:a2_1}
     \end{subfigure}
    \begin{subfigure}{0.34\textwidth}
         \centering
         \caption{}
         \includegraphics[width=\textwidth]{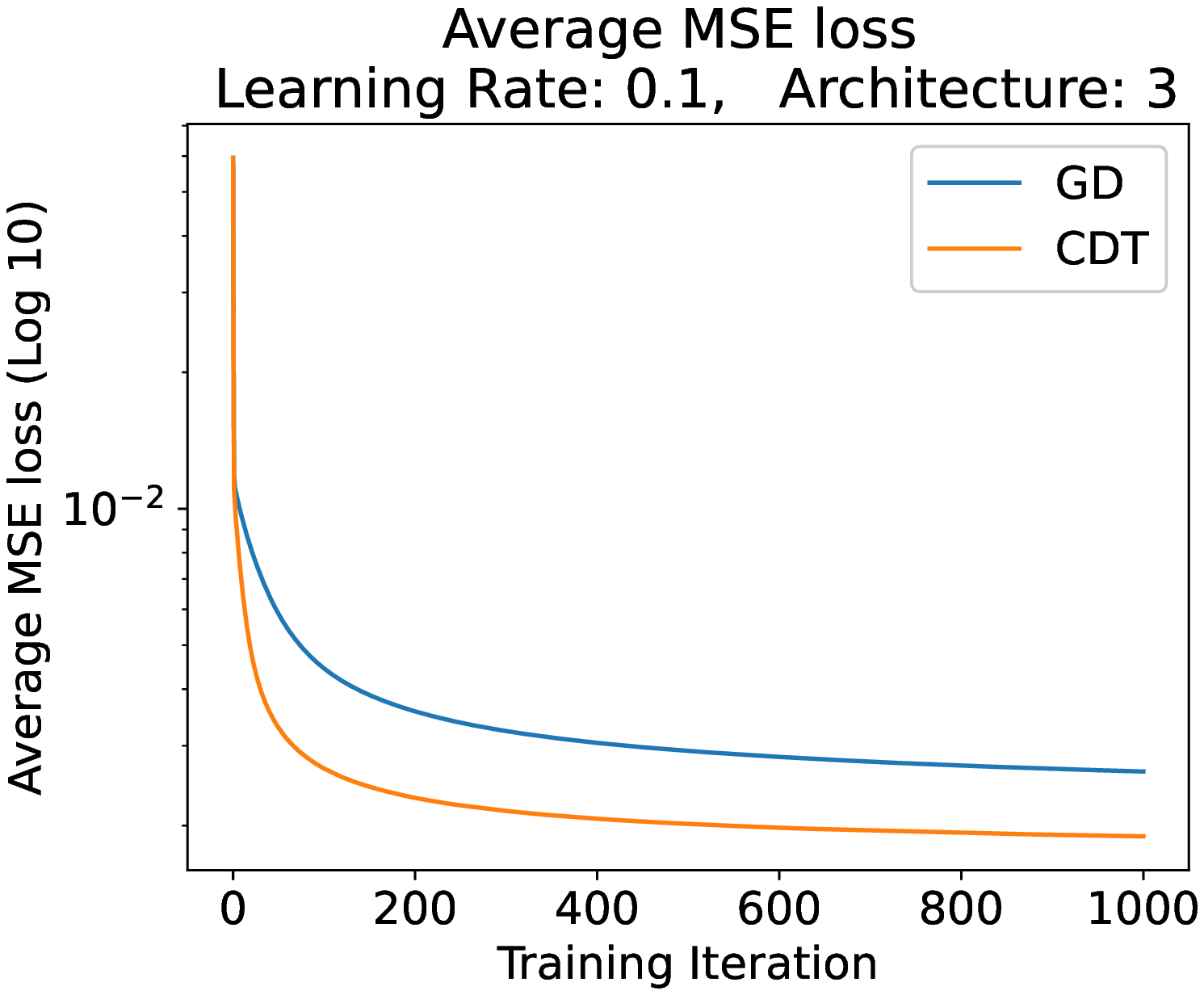}
         \label{fig:a2_01}
     \end{subfigure}
       \begin{subfigure}{0.34\textwidth}
         \centering
         \caption{}
         \includegraphics[width=\textwidth]{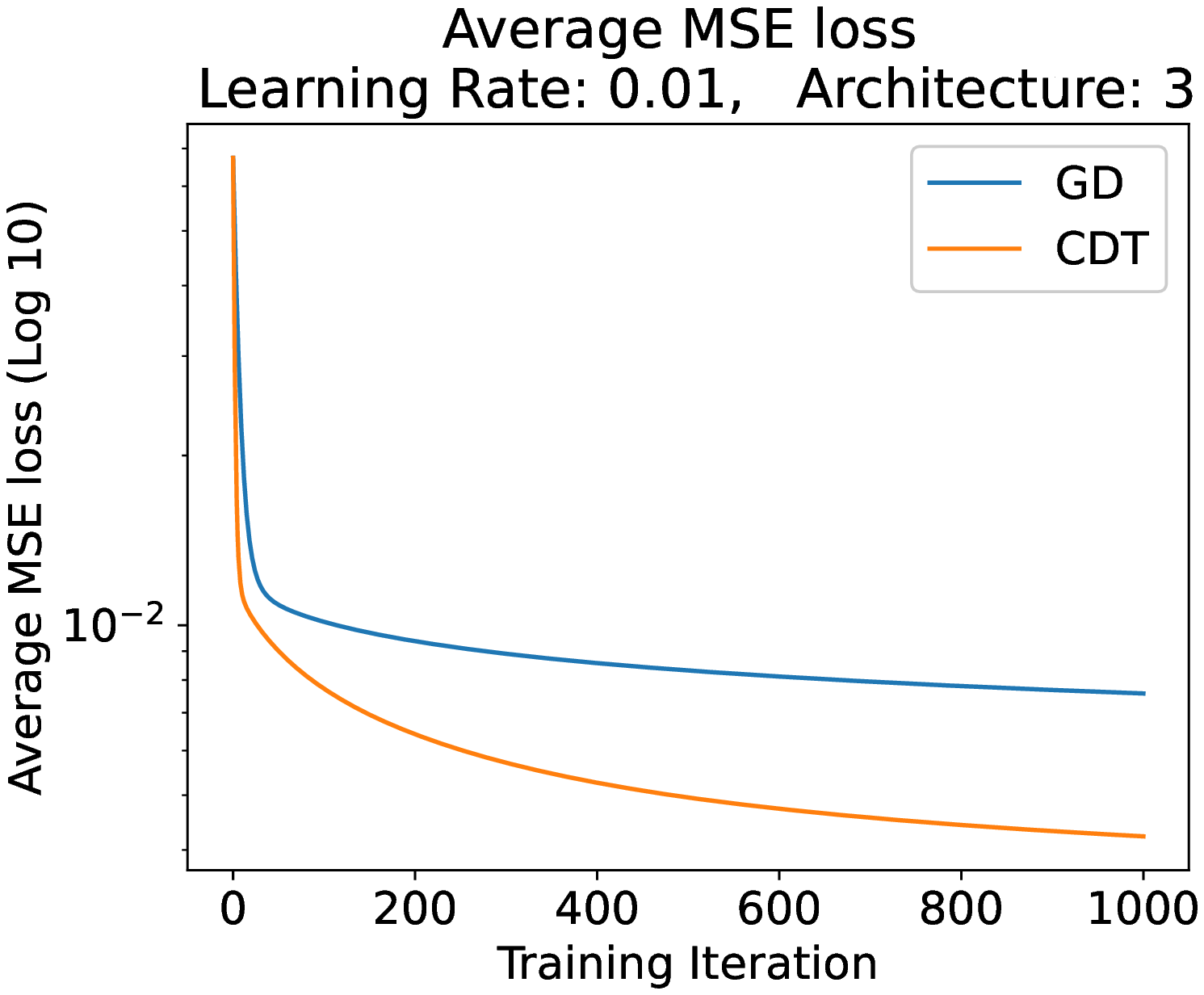}
         \label{fig:a2_001}
     \end{subfigure}
     \centering
    \begin{subfigure}{0.34\textwidth}
         \centering
         \caption{}
         \includegraphics[width=\textwidth]{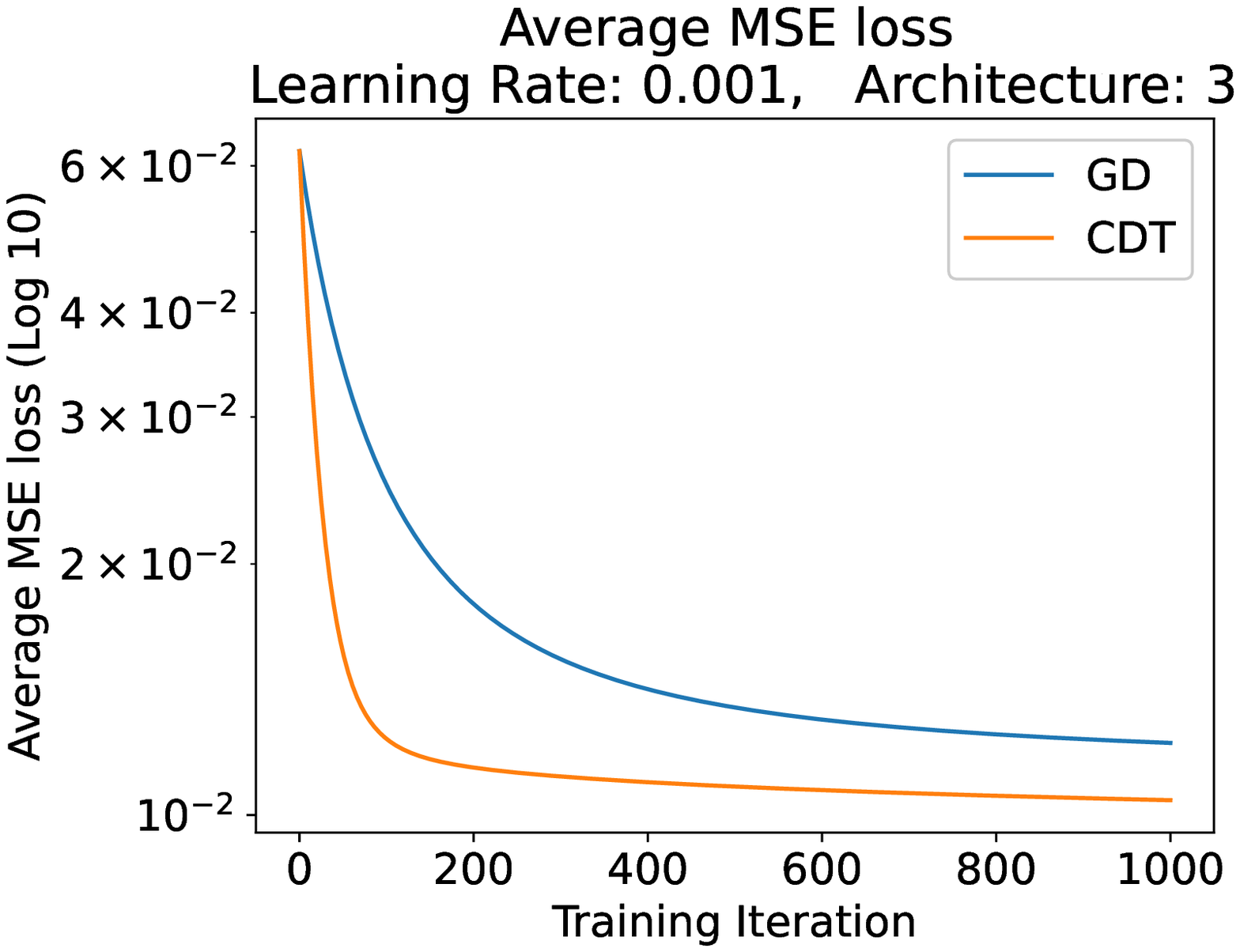}
         \label{fig:a2_0001}
     \end{subfigure}
     \caption{Absolute MSE loss on validation data for fully connected ANN architecture 3 averaged over all  10 initializations and data shuffles during training with GD and CDT. CDT converges with fewer iterations than GD for all learning rates. Moreover, CDT is stable at higher learning rates ($\alpha\geq1$), while GD diverges to $\infty_+$.}
\end{figure}

\newpage

\bibliography{Bibliography}

\section{Tables}
\begin{table}[h]
    \centering
    \begin{tabular}{P{2cm}|P{3.5cm}|P{4cm}}
        Architecture & Hidden Layers ($L-1$) & Width ($n_{l+1}$ $\forall l \in \{l_0,l_{L-1}\}$) \\
        \hline
        1 & 1 & 1500 \\
        2 & 3 & 500 \\
        3 & 6 & 250 \\
    \end{tabular}
    \caption{Single-target fully connected ANN architectures used for the regression experiments.}
    \label{tab:st-arch}
\end{table}

\begin{table}[!htp]
\begin{center}
\begin{tabularx}{0.8\textwidth}{P{1cm}l|P{1.5cm}P{2.5cm}P{4cm}P{2.5cm}}
$\alpha$& & Reachability & $|eig(\Theta(k_0))|<1$ &  Convergence  & Final validation loss (MSE $10^{-3}$) \\ \hline
1.000 &  GD  & Yes & No    &   No initialization &  $\infty \pm \infty$ \\ 
1.000 & CDT  & Yes  & No   &     All initializations & $150.01 \pm 350.50$ \\
\hline
0.100 & GD  & Yes & No &     Some initializations (7/10) & $3.04 \pm 0.59$ \\
0.100 & CDT & Yes & No   &     All initializations & $1.95 \pm 0.24$ \\
\hline
0.010 & GD  & Yes & Yes   &   All initializations & $2.70 \pm 0.61$ \\
 0.010 & CDT & Yes &  Yes  &  All initializations & $2.34 \pm 0.49$ \\
\hline
0.001 & GD  & Yes &Yes   &   All initializations & $5.90 \pm 0.95$ \\
 0.001 & CDT & Yes & Yes   &   All initializations & $3.72 \pm 0.76$

\end{tabularx}
\caption{Analytical and observed properties for fully connected architecture 1 on regression dataset.}
\label{table:a1}
\end{center}
\end{table}

\begin{table}[!htp]
\begin{center}
\begin{tabularx}{0.8\textwidth}{P{1cm}l|P{1.5cm}P{2.5cm}P{2.5cm}P{2.5cm}}
$\alpha$& & Reachability & $|eig(\Theta(k_0))|<1$ &  Convergence  & Final validation loss (MSE $10^{-3}$) \\ \hline
1.000 &  GD  & Yes & No    &   No initializations &  $\infty \pm \infty$ \\ 
1.000 & CDT  & Yes  & No   &    All initializations & $2.14 \pm 0.28$ \\
\hline
0.100 & GD  & Yes & No &  All initializations  & $2.31 \pm 0.45$ \\
0.100 & CDT & Yes & No   &    All initializations & $1.90 \pm 0.29$ \\
\hline
0.010 & GD  & Yes & Yes   &  All initializations & $3.71 \pm 0.62$ \\
 0.010 & CDT & Yes &  Yes  & All initializations & $2.62 \pm 0.54$ \\
\hline
0.001 & GD  & Yes &Yes   &  All initializations & $7.47 \pm 1.54$ \\
 0.001 & CDT & Yes & Yes   &  All initializations & $5.24 \pm 0.99$

\end{tabularx}
\caption{Analytical and observed properties for fully connected architecture 2 on regression dataset.}
\label{table:a2}
\end{center}
\end{table}

\begin{table}[!htp]
\begin{center}
\begin{tabularx}{0.8\textwidth}{P{1cm}l|P{1.5cm}P{2.5cm}P{2.5cm}P{2.5cm}}
$\alpha$& & Reachability & $|eig(\Theta(k_0))|<1$ &  Convergence  & Final validation loss (MSE $10^{-1}$) \\ \hline
1.000 &  GD  & Yes & No    &   No initializations &  $\infty \pm \infty$ \\ 
1.000 & CDT  & Yes  & No   &    All initializations & $1.88 \pm 0.27$ \\
\hline
0.100 & GD  & Yes & No &    All initializations & $2.63 \pm 0.80$ \\
0.100 & CDT & Yes & No   &    All initializations & $1.90 \pm 0.48$ \\
\hline
0.010 & GD  & Yes & Yes   &  All initializations & $7.58 \pm 1.84$ \\
 0.010 & CDT & Yes &  Yes  & All initializations & $4.23 \pm 1.05$ \\
\hline
0.001 & GD  & Yes &Yes   &  All initializations & $12.20 \pm 2.61$ \\
 0.001 & CDT & Yes & Yes   &  All initializations & $10.42 \pm 2.11$

\end{tabularx}
\caption{Analytical and observed properties for fully connected architecture 3 on regression dataset.}
\label{table:a3}
\end{center}
\end{table}

\begin{table}[!htp]
\begin{center}
\begin{tabularx}{0.8\textwidth}{P{1cm}l|P{1.5cm}P{2.5cm}P{2.5cm}P{2.5cm}}
$\alpha$& & Reachability & $|eig(\Theta(k_0))|<1$ &  Convergence  & Final validation loss (MSE $10^{-1}$) \\ \hline
1.000 &  GD  & Yes & No    &   No initializations &  $\infty \pm \infty$ \\ 
1.000 & CDT  & Yes  & No   &    All initializations & $2.59 \pm 0.15$ \\
\hline
0.100 & GD  & Yes & Yes &    All initializations & $3.37 \pm 0.38$ \\
0.100 & CDT & Yes & Yes &    All initializations & $3.33 \pm 0.37$ \\
\hline
0.010 & GD  & Yes & Yes   &  All initializations & $2.58 \pm 0.11$ \\
 0.010 & CDT & Yes &  Yes  & All initializations & $2.69 \pm 0.29$ \\
\hline
0.001 & GD  & Yes &Yes   &  All initializations & $2.67 \pm 0.10$ \\
 0.001 & CDT & Yes & Yes   &  All initializations & $2.58 \pm 0.11$

\end{tabularx}
\caption{Analytical and observed properties of ALEXNet on classification dataset.}
\label{table:alex}
\end{center}
\end{table}
\begin{table}
\centering

\begin{tabular}{|l|l|l|l|} 
\hline
 & Standard init. & NTK init.  & Improved standard  \\ 
\hline
Weight initialization  & $\mathcal{N}(0, \frac{\sigma_w^2}{s n_l})$ & $\frac{\sigma_w}{\sqrt{s n_l}}\mathcal{N}(0,1)$ & $\frac{1}{\sqrt{s}}\mathcal{N}(0, \frac{\sigma_w^2}{n_l})$  \\ 
\hline
Weight initialization  (conv.)  & $\mathcal{N}(0, \frac{\sigma_w^2}{s n_l n_m})$ & $\frac{\sigma_w}{\sqrt{s n_l n_m}}\mathcal{N}(0,1)$ & $\frac{1}{\sqrt{s}}\mathcal{N}(0, \frac{\sigma_w^2}{n_l n_m})$  \\ 
\hline
Bias initialization    & $\mathcal{N}(0, \sigma_b^2)$ & $\mathcal{N}(0, \sigma_b^2)$ & $\mathcal{N}(0, \sigma_b^2)$ \\
\hline
\end{tabular}
\caption{Different ANN weight and bias initializations}
\label{tab:NTK_init}
\end{table}

\end{document}